\documentclass{amsart}
\usepackage[utf8]{inputenc}
\input cyracc.def

\usepackage{amssymb}
\usepackage{amsthm}

\usepackage{tikz-cd}

\theoremstyle{plain}
\newtheorem{thm}{Theorem}[section]
\newtheorem{lem}[thm]{Lemma}
\newtheorem{cor}[thm]{Corollary}
\newtheorem{por}[thm]{Porism}
\newtheorem{prop}[thm]{Proposition}

\theoremstyle{definition}
\newtheorem{dfn}[thm]{Definition}
\newtheorem{Q}[thm]{Question}

\newtheorem{claim}[thm]{Claim}

\newtheorem{fact}[thm]{Fact}
\newtheorem{observation}[thm]{Observation}
\newtheorem{notation}[thm]{Notation}

\def\dnfo{\;\raise.2em\hbox{$\mathrel|\kern-.9em\lower.4em\hbox
{$\smile$}$}}

\def\dnf#1{\lower.9em\hbox{$\buildrel\dnfo\over{ \scriptstyle  #1}$}}

\def\dfo{\;\raise.2em\hbox{$\mathrel|\kern-.9em\lower.4em\hbox{$\smile$}
\kern-.72em\lower.07em\hbox{\char'57}$}\;}
\def\df#1{\lower1em\hbox{$\buildrel\dfo\over{\scriptstyle #1}$}}

\newcommand{\bQ}{\mathbb{Q}}

\newcommand{\sA}{\mathcal{A}}
\newcommand{\sB}{\mathcal{B}}

\newcommand{\sH}{\mathcal{H}}
\newcommand{\sI}{\mathcal{I}}

\newcommand{\sL}{\mathcal{L}}

\newcommand{\sO}{\mathcal{O}}

\newcommand{\sT}{\mathcal{T}}

\newcommand{\sbar}{\overline{\sigma}}
\newcommand{\tbar}{\overline{\tau}}
\newcommand{\rbar}{\overline{\rho}}

\newcommand{\abar}{\overline{\alpha}}
\newcommand{\bbar}{\overline{\beta}}
\newcommand{\gbar}{\overline{\gamma}}
\renewcommand{\dbar}{\overline{\delta}}

\newcommand{\diverge}{\!\!\uparrow}
\newcommand{\converge}{\!\!\downarrow}
\newcommand{\cat}{\widehat{\phantom{\alpha}}}

\newcommand{\tleq}{\trianglelefteq}

\newcommand{\seq}[1]{\langle #1\rangle}
\newcommand{\uhr}{\!\!\upharpoonright}
\newcommand{\dom}{\text{dom}}

\newcommand{\sigmaii}{\Sigma^1_1}

\newcommand{\piii}{\Pi^1_1}

\newcommand\+[1]{\mathcal{#1}}
\renewcommand\*[1]{\mathbf{#1}}

\begin{document}

\title{Structural Highness Notions}

\author[Calvert]{Wesley Calvert}
\address[Calvert]{School of Mathematical and Statistical Sciences\\ Mail Code 4408\\ 1245 Lincoln Drive\\ Southern Illinois University\\ Carbondale, Illinois 62901\\USA}
\email{wcalvert@siu.edu}
\urladdr{http://lagrange.math.siu.edu/calvert}

\author[Franklin]{Johanna N.Y.\ Franklin}
\address[Franklin]{Department of Mathematics \\ Room 306, Roosevelt Hall \\ Hofstra University \\ Hempstead, NY 11549-0114 \\ USA}
\email{johanna.n.franklin@hofstra.edu}
\urladdr{http://www.johannafranklin.net}
\thanks{The second author was supported in part by Simons Foundation Collaboration Grant \#420806.  Portions of this material are based upon work supported by the National Science Foundation under Grant No.\ DMS-1928930 while the first author participated in a program hosted by the Mathematical Sciences Research Institute in Berkeley, California, during
the Fall 2020 semester.}

\author[Turetsky]{Dan Turetsky}
\address[Turetsky]{School of Mathematics and Statistics \\ Victoria University of Wellington \\ Wellington \\ New Zealand}
\email{dan.turetsky@vuw.ac.nz}

\date{\today}

\begin{abstract}
    We introduce several highness notions on degrees related to the problem of computing isomorphisms between structures, provided that isomorphisms exist.  We consider variants along axes of uniformity, inclusion of negative information, and several other problems related to computing isomorphisms.  These other problems include Scott analysis (in the form of back-and-forth relations), jump hierarchies, and computing descending sequences in linear orders.
\end{abstract}

\maketitle
\tableofcontents

\section{Introduction}

The concepts of lowness and highness allow us to describe the power of an oracle. While they were originally established in the context of degree theory by Soare in the early 1970s \cite{s72}, they have been generalized to other areas of computability theory. In this paper, we investigate the application of highness to computable structure theory.

A set is low in a given setting if using it as an oracle in that context yields results no different from those obtained by using a computable set as an oracle. For instance, a set $A$ is low in degree theory if $A'\equiv_T 0'$, and it is low for a particular randomness notion if it cannot derandomize any random set \cite{flowhigh}. Lowness has even been studied in the context of learning theory \cite{ss91}.

Highness, on the other hand, has been studied in fewer settings. To define a lowness notion, we only need a relativizable class such as the $\Delta^0_2$ sets or the Martin-L\"{o}f random reals. To define a highness notion, we need some kind of maximality as well. In the case of the $\Delta^0_2$ sets, this role is played by $0'$; in the case of the Martin-L\"{o}f random reals, we must turn to the idea of highness for pairs of randomness notions and compare one class to another class containing it \cite{fsy-bases} since there is no natural ``maximally random" real.

In this paper, we discuss computable structure theory. Lowness was first studied in this setting by Franklin and Solomon in \cite{fs-lowim} as \emph{lowness for isomorphism}: a degree $\* d$ is low for isomorphism if, for any two computable structures $\sA$ and $\sB$, the degree $\* d$ only computes an isomorphism between $\sA$ and $\sB$ if there is already a computable isomorphism between them.\footnote{Csima has also introduced the notion of \emph{lowness for categoricity} \cite{csimatalk}, though it has not been investigated in much depth as of yet.} 

The class of degrees that are low for isomorphism defies easy characterization. No degree comparable to $\* 0'$ is low for isomorphism except $\* 0$.  On the other hand, every 2-generic degree is \cite{fs-lowim}. Further work was done on this class by Franklin and Turetsky \cite{ft-1genlow} and then, once again, Franklin and Solomon \cite{fs-lowim20}, but these papers only suggest that the class is more complicated than had been initially thought. Nevertheless, it seems to be somewhat robust: Franklin and Turetsky have characterized these degrees as those that are low for paths \cite{ft-lowpaths}, and Franklin and McNicholl have characterized them as those that are low for isometric isomorphism \cite{fm-lowisom}.

Turning to highness in the context of computable structure theory, we say that a degree is high for isomorphism if it can compute an isomorphism between any two isomorphic computable structures.

\begin{dfn}
We call a degree $\* d$ {\em high for isomorphism} if for any two computable structures $\+ M$ and $\+ N$ with $\+ M \cong \+ N$, there is a $\* d$-computable isomorphism from $\+ M$ to $\+ N$.
\end{dfn}

Note that the degree of Kleene's $\+ O$ is high for isomorphism.  Indeed, $\+ O$ can compute a path through the tree of partial isomorphisms of any two computable structures for which such a path exists.  Thus, in analogy with the  $\Delta^0_2$ degrees, $\+O$ is the maximal object for isomorphism.

For a set $X$, we will take \[\+ O^X = \{ e : \text{$\{e\}^X$ is the graph of a binary relation that well-orders $\omega$}\}\] and $\+ O = \+ O^\emptyset$.  By standard results described, for instance, in \cite{Sacks1990}, this is equivalent to the usual definition involving notations for ordinals.

We also define a uniform version of highness for isomorphism in keeping with uniform versions of other concepts in computable structure theory \cite{dkk03}.

\begin{dfn}
We call a degree $\* d$ {\em uniformly high for isomorphism} if there is a $D \in \* d$ and a total computable $f$ such that for any computable structures $\+M_i \cong \+M_j$, the function $\{f(i,j)\}^D$ is an isomorphism from $\+M_i$ to $\+M_j$.
\end{dfn}

The primary focus of this paper is highness for isomorphism, but we also consider several other structural highness notions.  We begin in Section \ref{sec:high} with a discussion of the types of degrees that are high for isomorphism. In Section \ref{sec:scott}, we turn to Scott analysis in an attempt to characterize these degrees further, and in Section \ref{sec:descending}, we consider degrees that are high for descending sequences, that is, those that can compute a descending sequence in any computable ill-founded linear order. In Section \ref{sec:jumps}, we relate these notions to degrees that can compute jump structures on arbitrary Harrison orders. We then consider a structural equivalent of enumerability, ``reticence," and the reticent versions of all of the previously considered notions in Section \ref{sec:reticence}. In Section \ref{sec:rand_gen}, we make some brief remarks on the relationships between highness for isomorphism and randomness and genericity, and we conclude the paper with some open questions in Section \ref{sec:etc}.

\section{Highness and uniform highness for isomorphism}\label{sec:high}

We begin by noting the following connection between the degrees that are high for isomorphism and those that are uniformly high for isomorphism.

\begin{prop}\label{prop:double_high_is_uniform}
If $\* d$ is high for isomorphism, then $\* d''$ is uniformly high for isomorphism. 
\end{prop}
\begin{proof}
Observe that $\{e\}^D$ being an isomorphism between $\+M_i$ and $\+M_j$ is a $\Pi^0_2(D)$ property of $e, i$ and $j$: being total and surjective are $\Pi^0_2$, and being injective and being a partial isomorphism are $\Pi^0_1$.  Thus, given $i$ and $j$, $D''$ can iterate through the $e$s until it finds one giving an isomorphism and then output $\{e\}^D$.
\end{proof}

Now we turn our attention to establishing a characterization of the degrees that are (uniformly) high for isomorphism and proceed to some existence results.

We observe that if $\* d$ is not high for isomorphism, then there is a $\sigmaii$ class of which it computes no element: the class of isomorphisms between some pair of isomorphic computable structures.\footnote{By a $\Sigma^1_1$ class, we mean a $\Sigma^1_1$ subset of $\omega^\omega$.}  This leads to the following definition inspired by the characterization of a $\Pi^0_1$ class as the set of paths through some computable tree. 

\begin{dfn}
We call a degree $\* d$ {\em high for paths} if for every nonempty $\Pi^0_1$ class of functions $\+ P$, the degree $\* d$ computes an element of $\+ P$. Similarly, a degree $\* d$ is called {\em uniformly high for paths} if there is a $D \in \* d$ and a total computable $f$ such that for every nonempty $\Pi^0_1$ class of functions $\+ P_i$, the function $\{f(i)\}^D$ is an element of $\+ P_i$.
\end{dfn}

As $\Sigma^1_1$ classes are uniformly projections of $\Pi^0_1$ classes in the sense that any $\Sigma^1_1$ class can be viewed as the projection on the first coordinate of the set of paths through a $\Pi^0_1$ subtree of $\left(\omega^\omega\right)^2$, we could replace $\Pi^0_1$ with $\Sigma^1_1$ in this definition without changing the class of degrees described. This gives us the following:

\begin{observation}
If $\* d$ is (uniformly) high for paths, then it (uniformly)  computes an element of every nonempty $\Sigma^1_1$ class, and conversely.
\end{observation}

This allows us to make use, from time to time, of the following result (described, for instance, in \cite{Sacks1990}).

\begin{thm}[Gandy Basis Theorem] Every nonempty $\Sigma^1_1$ class has a member $f$ such that $\omega_1^f = \omega_1^{ck}$. \end{thm}

We now present a characterization of the degrees that are high for isomorphism in terms of highness for paths; we will need the following notation for the proof.

\begin{notation}\label{TreeStar}
If $S, T \subseteq \omega^{<\omega}$ are trees, we define $S\star T = \{ (\sigma_0, \sigma_1) : \sigma_0 \in S \ \& \ \sigma_1 \in T \ \& \ |\sigma_0| = |\sigma_1|\}$ and give this a tree structure by $(\sigma_0, \sigma_1) \subseteq (\tau_0, \tau_1)$ if $\sigma_i \subseteq \tau_i$ for $i < 2$.  
\end{notation}

Note that we can identify $S\star T$ with a tree in $\omega^{<\omega}$ in an effective fashion.

\begin{prop}\label{prop:high_for_iso_vs_paths}
A degree is (uniformly) high for isomorphism if and only if it is (uniformly) high for paths.
\end{prop}

\begin{proof}
Suppose that $\* d$ is (uniformly) high for paths.  Given two isomorphic computable structures, we construct the tree $\sI$ of partial isomorphisms, noting that this tree is uniformly computable in the two structures.  Now $\* d$ computes a path through $\sI$, which is an isomorphism.

The other direction follows from the proof of Theorem~\ref{thm:highforisoandharrison}, but we give a proof here as well.  The reader familiar with the interplay between trees and linear orders may recognize that these are substantially the same proof, but as trees are more general objects than linear orders, this version is somewhat simpler.  Let $\*d$ be (uniformly) high for isomorphism, and let $\sT$ be a nonempty $\Pi^0_1$ class of functions generated by a computable tree $T$.

Our proof is based on rank-saturated trees of infinite rank \cite{CalvertKnightMillar2006,FFHKMM}; we will require the following facts:
\begin{enumerate}
    \item There is a computable rank-saturated tree $S$ of infinite rank that has a computable path.  (This follows from the proof of Lemma 1 in \cite{FFHKMM}, beginning with a Harrison order with a computable descending sequence; such an order is constructed in the proof of Theorem~\ref{thm:highforisoandharrison}.) 
    \item Any two computable rank-saturated trees of infinite rank are isomorphic.  (Proposition 2, \cite{FFHKMM})
    \item If $S$ is a computable rank-saturated tree of infinite rank and $T$ is any computable tree with a path, then $T\star S$ is a computable rank-saturated tree of infinite rank.  (Proposition 1, \cite{FFHKMM})
\end{enumerate}

Then $S$ and $T\star S$ are our two computable structures in some appropriate language.  Let $f$ be a computable path through $S$.  Since $\*d$ can compute an isomorphism between $S$ and $T\star S$, we have that $\*d$ can carry $f$ over to a path in $T\star S$ and then project that down to a path in $T$ as required.  Furthermore, this process is uniform in the isomorphism.
\end{proof}

\subsection{Relationship between highness for isomorphism and benchmark Turing degrees}

We first note that the degrees that are high for isomorphism are necessarily strong:

\begin{prop}\label{prop:high_for_iso_above_hyp}
If $\* d$ is high for isomorphism, then $\* d$ computes every $\Delta^1_1$ set.
\end{prop}

\begin{proof}
If $X$ is $\Delta^1_1$, then $\{X\}$ is a $\Sigma^1_1$ class, and thus $\* d$ computes $X$.
\end{proof}

Another perspective reinforces this insight.

\begin{prop}\label{prop:jumps_of_high_are_powerful} Kleene's $\sO$ is arithmetical over any degree high (or uniformly high) for isomorphism. Furthermore, we have the following:
\begin{enumerate}
\item If $\*d$ is high for isomorphism, then $\+O$ is $\Pi^0_3(\*d)$ and thus $\*d^{(3)} \ge \+O$.
\item If $\*d$ is uniformly high for isomorphism, then $\+O$ is $\Sigma^0_2(\*d)$ and thus $\*d^{(2)} \ge \+O$.
\end{enumerate}
\end{prop}

\begin{proof}
Note that $X = \{ (i, j) : \+M_i \cong \+M_j\}$ is a $\Sigma^1_1$-complete set (folklore, see \cite{gk02}).  However, if $\*d$ is high for isomorphism with $D \in \*d$, then $(i,j) \in X$ if and only if $\exists e\, [\{e\}^D: \+M_i \cong \+M_j]$, and the matrix of the right-hand side is $\Pi^0_2(D)$.  Then $\Sigma^1_1$ sets are $\Sigma^0_3(\*d)$, making $\+O \in \Pi^0_3(\*d)$.

If, instead, $\*d$ is uniformly high for isomorphism as witnessed by $D$ and $f$, then we can strip off the opening existential quantifier from the previous argument: $(i, j) \in X$ if and only if $\{f(i,j)\}^D: \+M_i \cong \+M_j$.  In this case, $\Sigma^1_1$ sets are $\Pi^0_2(\*d)$, making $\+O \in \Sigma^0_2(\*d)$.
\end{proof}

We now progress toward showing the existence of degrees strictly below $\sO$ which are high for isomorphism.  To this end, we use the following known strengthening of the Gandy Basis Theorem:

\begin{lem}[Folklore, see Exercise 2.5.6 in~\cite{BookChongYu2015}]
If $\+ Q \subseteq \omega^\omega$ is a nonempty $\Sigma^1_1$ class, there is an $X \in \+Q$ with $\+ O^X \le_T \+ O$.  Furthermore, an index for a witnessing Turing reduction can be uniformly obtained from an index for $\+Q$.
\end{lem}

\begin{proof}
Let
\begin{align*}
\+ R = \{ (X, f) \in \+Q \times (\omega\cup \{-1\})^\omega &: \text{if $f(e) = -1$, then $\{e\}^X$ is ill-founded,}\\ & \text{and if $f(e) = n > -1$, then $\{e\}^X \cong \{n\}^\emptyset$}\}.
\end{align*}
Note that $\+R$ is $\Sigma^1_1$.  In order to show that $\+R$ is nonempty via the Gandy Basis Theorem, we fix $X \in \+Q$ with $\omega_1^X = \omega_1^{ck}$.  Then for every $e$ such that $\{e\}^X$ is well founded, there is an $n$ with $\{e\}^X \cong \{n\}^\emptyset$, so we define $f(e) = n$ for such an $n$.  For $e$ such that $\{e\}^X$ is ill founded, we define $f(e) = -1$.  Then $(X, f) \in \+R$.

As $\+R$ is a nonempty $\Sigma^1_1$ class, there is some $(X, f) \in \+R$ computable from $\+O$.  Then $e \in \+O^X$ if and only if  $-1 < f(e) \in \+O$, and so $\+O^X \le_T \+O$.  Furthermore, the process for constructing the Turing reduction is uniform from an index for $\+Q$.
\end{proof}

Since there is a Turing functional $\Phi$ with $\Phi^{\+ O^X} = X$ for all $X$, we immediately get $X \le_T \+ O$ uniformly for the $X$ of the above lemma.  This also follows from the construction.

Note that our construction relativizes to provide the following:
\begin{lem}
For any oracle $A$, if $\+Q$ is a non-empty $\Sigma^1_1(A)$ class, there is an $X \in \+Q$ with $X \le_T \+O^A$ and $\+O^{A,X} \le_T \+ O^A$, and indices for these reductions can be found uniformly from an index for $\+Q$ independently of $A$.
\end{lem}

We will also need some basic results about pointed perfect trees. 

\begin{dfn}
A {\em pointed perfect tree} is a function $f: 2^{<\omega} \to 2^{<\omega}$ such that $\sigma \subseteq \tau$ if and only if $f(\sigma) \subseteq f(\tau)$ and such that for every $X \in 2^\omega$, we have that $f(X) = \bigcup_n f(X\uhr_n)$ computes $f$.

A {\em uniformly pointed perfect tree} is a pointed perfect tree $f$ such that there is a Turing functional $\Phi$ with $\Phi^{f(X)} = f$ for all $X \in 2^\omega$.

If $f$ is a pointed perfect tree and $Y \in 2^\omega$, we define $g = f\oplus Y$ by
\[
g(\sigma) = f(\sigma \oplus Y\uhr_{|\sigma|}), 
\]
and if $f$ and $g$ are pointed perfect trees, we define $g \le f$ if $\text{range}(g) \subseteq \text{range}(f)$.
\end{dfn}

Some explanation may be in order on the terminology.  The literal ``tree'' in question is the downward closure of the range of the function $f$, and under the conditions described, this tree is, in fact, perfect in the sense that there is splitting beyond any finite node.  In several older papers the authors have consulted that use this terminology, there seems to be no clear explanation of why the term ``pointed" has the meaning it is given here.

\begin{lem}
If $f$ is a (uniformly) pointed perfect tree, then the following hold:
\begin{enumerate}
    \item For any $Y \in 2^\omega$, $g = f\oplus Y$ is a (uniformly) pointed perfect tree,
    \item $Y \le_T g(X)$ for every $X \in 2^\omega$, and 
    \item $Y \le_T g$.
\end{enumerate}\end{lem}
\begin{proof}
Certainly $g$ is a function with the necessary extension condition.  It remains in point 1 to show the (uniform) computability of $g$ from each $g(X)$.  For any $X$, we have $g(X) = f(X\oplus Y)$, so $g(X)$ computes $f$ by assumption.  The pair $(g(X),f) = (f(X\oplus Y), f)$ computes $X\oplus Y$ and so computes $Y$, and $(f,Y)$ computes $g$.  Since $g$ computes $g(0^\infty)$, it follows that $Y \le_T g$.  All of this is uniform if the reduction from $f(X\oplus Y)$ to $f$ is.
\end{proof}

Note that the formula that says ``$f$ is a uniformly pointed perfect tree'' is arithmetic (the argument relies on the compactness of $2^\omega$).  Note also that if $f$ is a pointed perfect tree, $|f(\sigma)| \ge |\sigma|$ for all $\sigma$.

We are now ready to construct a high for isomorphism degree strictly below $\+O$. Our proof is inspired by Jockusch and Simpson's construction in \cite{JockuschSimpson1976} of a minimal upper bound for $\Delta^1_1$ with a triple jump computable from $\+O$, although we use strongly hyperlow trees (i.e., trees $f$ such that $\+O^f \le_T \+O$) rather than their $\Delta^1_1$ trees.

\begin{prop}
There is a degree $\* d$ which is high for isomorphism with $\* d^{(3)} = \+ O$.
\end{prop}

\begin{proof}

For a set $X$, we take $X^{(3)} = \{ e : \exists m\, [\lambda z.\{e\}^X(m,z) \text{ is total}]\}$.

We build an $\+O$-computable sequence of uniformly pointed perfect trees $f_0 \ge f_1 \ge \dots$ where $\+ O^{f_i} \le_T \+ O$ for each $i$ (and the sequence of indices for the reductions is computable in $\+O$).  We alternate between working to force the triple jump and coding the next required isomorphism.

We begin by setting $f_0$ to be the identity: $f_0(\sigma) = \sigma$.

\smallskip

Suppose we have defined $f_n$ with $n = 2\seq{e,0}$.  Consider the class
\begin{align*}
U_e = \{ g : \ &\text{$g$ is a uniformly pointed perfect tree} \ \& \ g \le f_n\\ &\& \ \exists m\, [\lambda z.\{e\}^{g(X)}(m,z)\text{ is total for every $X \in 2^\omega$}]\}.
\end{align*}
This is arithmetic relative to $f_n$ (again, this relies on compactness) and thus is a $\Sigma^1_1(f_n)$ class.  Thus $\+O^{f_n}$ can decide whether $U_e$ is empty, and since $\+O^{f_n} \le_T \+O$, so can $\+O$.

If $U_e$ is nonempty, there is a $g \in U_e$ with $\+O^{g} \le_T \+O^{f_n} \le_T \+O$, and indices for the relevant reductions can be obtained effectively.  We thus let $f_{n+1} = g$, and we have now forced $e$ into the triple jump.  By replacing $g$ with $\lambda \sigma. g(0^n\cat\sigma)$, we may assume that $|f_{n+1}(\seq{})| > n$.

If $U_e$ is empty, we let $f_{n+1} = f_n$ (again, adjusting to ensure $|f_{n+1}(\seq{})| > n$).  By our actions at steps $2\seq{e,m+1}$ for $m \in \omega$, we will force $e$ out of the triple jump.

\smallskip

Now suppose instead that we have defined $f_n$ with $n = 2\seq{e, m+1}$.  If $U_e$ was nonempty at step $2\seq{e,0}$, then the triple jump is already forced at $e$, so we let $f_{n+1} = f_n$.

If $U_e$ was empty at step $2\seq{e,0}$, we know there are $\tau$ and $z$ such that $\{e\}^{f_n(\sigma)}(m,z)\diverge$ for every $\sigma \supseteq \tau$: if this were not the case, we could define an $h$ as follows for a contradiction.
\begin{itemize}
\item Let $h(\seq{}) = \sigma$ for the first string $\sigma$ discovered with $\{e\}^{f_n(\sigma)}(m,0)\converge$.
\item For $\rho \in 2^{<\omega}$ and $i < 2$, let $h(\rho\cat i) = \sigma$ for the first string $\sigma \supseteq h(\rho)\cat i$ discovered with $\{e\}^{f_n(\sigma)}(m,|\rho|+1)\converge$.
\end{itemize}
By assumption, $h$ is total.  Now we let $g = f_n\circ h$.  Note that $h \le_T f_n$, so $g \le_T f_n$, and for every $X$, $g(X) = f(h(X))$.  Thus $g$ is a uniformly pointed perfect tree, and $g \le f_n \le f_{2\seq{e,0}}$.  Also, $\lambda z.\{e\}^{g(X)}(m,z)$ is total for all $X$, and so $g \in U_e$ contrary to assumption.  Therefore, there must be such a $\tau$ and $z$.  The set of such pairs $(\tau, z)$ is arithmetical in $f_n$, and so $\+O^{f_n} \le_T \+O$ can uniformly find such a pair.

We define $f_{n+1}(\sigma) = f_n(\tau\cat \sigma)$.  So $\{e\}^{f_{n+1}(X)}(m,z)\diverge$ for every $X \in 2^\omega$, and we have taken another step towards forcing the triple-jump at $e$.

\smallskip

Suppose instead we have defined $f_n$ with $n = 2\seq{k,m}+1$.  We will choose $f_{n+1}$ able to compute an isomorphism between $\+ M_k$ and $\+ M_m$ if such an isomorphism exists.  Consider $V_n = \{ r\in \omega^\omega \ | \ r : \+ M_k \cong \+ M_m\}$.  This is a $\Sigma^1_1$ class, and $\+O$ can therefore determine whether it is empty.

If $V_n$ is empty, we define $f_{n+1} = f_n$. If $V_n$ is nonempty, then we note that $V_n$ is also a $\Sigma^1_1(f_n)$ class.  This means that there is an $r \in V_n$ with $\+ O^{r,f_n} \le_T \+ O^{f_n} \le_T \+ O$, and relevant indices can be obtained effectively.  Let $R$ be a set coding $r$ in some effective fashion, and let $f_{n+1} = f_n\oplus R$.  Then $f_{n+1} \le_T (f_n, R)$, and so $\+ O^{f_{n+1}} \le_T \+O$.

\smallskip

By construction, $|f_{n+1}(\seq{})| > n$ for every $n = 2\seq{e,0}$.  We let
\[
D = \bigcup_{n = 2\seq{e,0}} f_{n+1}(\seq{}).
\]
Then for every $n$, $D = f_n(X)$ for some $X$, so if $\+M_k \cong \+M_m$, we can let $n = 2\seq{k,m}$.  Furthermore, for the $r$ chosen at step $n$, we have \[r \equiv_T R \le_T f_n\oplus R = f_{n+1} \le_T D,\] and thus $D$ computes an isomorphism between $\+M_k$ and $\+M_m$.

Furthermore, $e \in D^{(3)}$ precisely when $U_e$ is nonempty by construction.  Since the construction is $\+O$-computable, $D^{(3)} \le_T \+O$.  By Proposition~\ref{prop:jumps_of_high_are_powerful}, $\*d = \text{\bf deg}(D)$ is the desired degree.
\end{proof}

From this result, combined with Proposition \ref{prop:double_high_is_uniform}, we immediately obtain the following.

\begin{cor}
There is a degree $\* {d}$ which is uniformly high for isomorphism with $\* {d}' = \sO$.
\end{cor}

It is worthwhile to contrast the state of our knowledge for the jumps of degrees high for isomorphism with those uniformly high for isomorphism.  Certainly both classes are closed upwards in the Turing degrees. Proposition~\ref{prop:jumps_of_high_are_powerful} tells us that a degree high for isomorphism is at most three jumps below $\+O$, and we have an example showing that this is tight.  On the other hand, the same proposition tells us that a degree uniformly high for isomorphism is at most two jumps below $\+O$, but the best example we have is only one jump below $\+O$.  We can, at least, rule out one nonexample.

\begin{prop}
If $\*d$ is high for isomorphism with $\*d^{(3)} = \+O$, then $\*d'$ is not uniformly high for isomorphism.
\end{prop}

\begin{proof}
If $\*d'$ were uniformly high for isomorphism, then by Proposition~\ref{prop:jumps_of_high_are_powerful}, $\+O$ would be $\Sigma^0_2(\*d')$ and thus $\Sigma^0_3(\*d)$.  As $\*d$ is high for isomorphism, $\+O$ is $\Pi^0_3(\*d)$, and so $\+O$ would be $\Delta^0_3(\*d)$, contradicting $\*d^{(3)} = \+O$.
\end{proof}

It is perhaps natural to wonder whether $\sO$ is the only bound of all degrees below $\sO$ which are high for isomorphism.  We might also ask whether the degrees high for isomorphism are linearly ordered, or whether there is a least degree which is high for isomorphism.  The next result resolves all of these questions.

\begin{prop} There exist degrees $\mathbf{d}_1$ and $\mathbf{d}_2$, each uniformly high for isomorphism, with $\mathbf{d}_i \lneq_T \sO$ such that $\mathbf{d}_1 \oplus \mathbf{d}_2 \equiv_T \sO$, and any $\mathbf{d} \le_T \mathbf{d}_1, \mathbf{d}_2$ is hyperarithmetic.\end{prop}

\begin{proof}  

We will construct, effectively in $\sO$, two degrees $\mathbf{d}_1$ and $\mathbf{d}_2$, each of which enumerates $\omega - \sO$, and such that $\mathbf{d}_1 \oplus \mathbf{d}_2 \equiv_T \sO$.  By Theorem \ref{thm:reticence_all_the_same}, these degrees will be uniformly high for isomorphism.  

Let $\langle \cdot, \cdot \rangle$ be a bijection of $\omega^2$ with $\omega - \{2^k : k \in \omega\}$.  We first define an operator $R$ on $2^{<\omega}$ by setting $R(\sigma)$ to be the set of $n$ such that there is $k$ with $\sigma(\langle k,n \rangle) = 1$, and we set $\delta_{1,0}=\delta_{2,0} = \emptyset$.  We will maintain that $R(\delta_{1,s}), R(\delta_{2,s}) \subset \omega - \+O$, and we will arrange that $\bigcup_s R(\delta_{1,s}) = \bigcup_s R(\delta_{2,s}) = \omega - \+O$.

At stage $5s$, we search, using oracle $\sO$, for some $\sigma \in 2^{<\omega}$ extending $\delta_{1,5s}$ such that $\varphi_s^\sigma(s)\converge$ and $R(\sigma) \cap \sO = \emptyset$.  If such a $\sigma$ exists (a $\Sigma_1^1$ condition), we set $\delta_{1,5s+1} = \sigma$ and $\delta_{2,5s+1} = \delta_{2,5s} \cup \{(2^{k_{5s}},1)\}$ for the least $k_{5s} >\left|\delta_{2,5s}\right|$.  Otherwise, we set $\delta_{1,5s+1} = \delta_{1,5s}$ and $\delta_{2,5s+1} = \delta_{2,5s} \cup \{(2^{k_{5s}},0)\}$.  At stage $5s+1$, we act similarly, exchanging the roles of $\delta_1$ and $\delta_2$.

At stage $5s+2$, we let $n_{5s+1}$ be the least such that \[n_{5s+1} \in (\omega - \sO) \cap \left(\omega - R\left( \delta_{1,5s+2} \right)\right)\] and find the least $k_{5s+2}$ such that $\langle k_{5s+2},n_{5s+1} \rangle > \left|\delta_{1,5s+2}\right|$.  We set \[\delta_{1,5s+3} = \delta_{1,5s+2} \cup \left\{(\langle k_{5s+1},n_{5s+1} \rangle, 1)\right\}\] and $\delta_{2,5s+3} = \delta_{2,5s+2}$.  At stage $5s+3$, we act similarly, again exchanging the roles of $\delta_1$ and $\delta_2$.

At stage $5s+4$, we let $s = \seq{e_1, e_2}$.  Fix the least $k_1 > |\delta_{1, 5s+4}|$ and $k_2 > |\delta_{2, 5s+4}|$.  We search, using oracle $\sO$, for some $\sigma_1$ extending $\delta_{1, 5s+4}$ and $\sigma_2$ extending $\delta_{2, 5s+4}$ with $\sigma_1(2^{k_1}) = \sigma_2(2^{k_2}) = 1$ and $R(\sigma_1) \cap \sO = R(\sigma_2)\cap \sO = \emptyset$ and some $m$ with $\Phi_{e_1}^{\sigma_1}(m)\converge \neq \Phi_{e_2}^{\sigma_2}(m)\converge$.  If such $\sigma_1$ and $\sigma_2$ exist (a $\Sigma^1_1$ condition), we set $\delta_{1,5s+5} = \sigma_1$ and $\delta_{2, 5s+5} = \sigma_2$.

If no such pair exists, we instead search for a split above just $\delta_{1,5s+4}$.  If there are strings $\tau_1, \tau_2$ both extending $\delta_{1, 5s+4}$ with $\tau_1(2^{k_1}) = \tau_2(2^{k_1}) = 1$ and $R(\tau_1)\cap \sO = R(\tau_2)\cap \sO = \emptyset$ and some $m$ with $\Phi_{e_1}^{\tau_1}(m)\converge \neq \Phi_{e_1}^{\tau_2}(m)\converge$, then we let \[\delta_{1,5s+5} = \delta_{1, 5s+4} \cup \{(2^{k_1}, 0)\}\] and \[\delta_{2, 5s+5} = \delta_{2, 5s+4} \cup \{2^{k_2}, 1)\}.\]  Otherwise, we let $\delta_{1,5s+5} = \delta_{1, 5s+4} \cup \{(2^{k_1}, 1)\}$ and $\delta_{2, 5s+5} = \delta_{2, 5s+4} \cup \{(2^{k_2}, 0)\}$.

We now let $D_i = \bigcup\limits_{s \in \omega} \delta_{i,s}$ for each $i$.  We note that, by construction, each $D_i$ enumerates $(\omega - \sO)$ and is computable from $\sO$; indeed, $\sO$ can compute $D_i'$, since when we find at stage $5s$ that there is no $\sigma$ which both forces convergence and maintains disjointness from $\sO$, we commit to never allowing $\varphi_s^{D_i}(s)$ to converge.  By a standard argument, $D_1\oplus D_2$ can compute the sequence $(\delta_{1, s}, \delta_{2, s})_{s \in \omega}$.  Thus the join can compute $\+O$, as $n \in \+O$ if and only if $n \not \in R(\delta_{1, 5(n+1)})$.

Finally, we argue that if $X \le_T D_1, D_2$, then $X$ is hyperarithmetic.  Fix $s = \seq{e_1, e_2}$ with $X = \Phi_{e_1}^{D_1} = \Phi_{e_2}^{D_2}$.  Then we must not have found the desired strings $\sigma_1, \sigma_2$ at stage $5s+4$.  Suppose we found the subsequent strings $\tau_1, \tau_2$ and the corresponding $m$ instead.  Then $D_2(2^{k_2}) = 1$ and $\Phi_{e_2}^{D_2}(m)\converge$, so there is some initial segment of $D_2$ witnessing this convergence and extending $\delta_{2, 5s+4}$.  There is some $j$ such that $\Phi_{e_1}^{\tau_j}(m) \neq \Phi_{e_2}^{D_2}(m)$, which gives us a pair $\sigma_1, \sigma_2$, contrary to the above.

So we did not find $\tau_1$ and $\tau_2$, meaning that $D_1(2^{k_1}) = 1$.  Thus, for any $m$ and any $\tau$ extending $\delta_{1, 5s+4}$ with $\tau(2^{k_1}) = 1$ and $R(\tau)\cap \sO = \emptyset$, we have \[\Phi_{e_1}^{\tau}(m)\converge \Rightarrow \Phi_{e_1}^{\tau}(m) = \Phi_{e_1}^{D_1}(m) = X(m).\]  Thus $X(m) = 1$ if and only if there is such a $\tau$ with $\Phi_{e_1}^{\tau}(m) = 1$, and $X(m) = 0$ if and only if there is such a $\tau$ with $\Phi_{e_1}{\tau}(m)$. These are both $\Sigma^1_1$ conditions, so $X$ is $\Delta^1_1$.
\end{proof}

\subsection{High for a restricted class of isomorphisms}

In the definition of highness for isomorphism, we considered isomorphisms between any two computable structures.  We will now restrict our attention to isomorphisms of a particular class of computable structures, namely presentations of the Harrison order $\omega_1^{ck}\cdot (1+\bQ)$; a similar analysis was done in the context of lowness for isomorphism by Suggs \cite{suggs}.

\begin{dfn}
A degree $\*{d}$ is {\em high for isomorphism for Harrison orders} if $\*{d}$ computes an isomorphism between any two computable linear orders of order type $\omega_1^{ck}(1+\bQ)$.  We define \emph{uniformly high for isomorphism for Harrison orders} in the natural way.
\end{dfn}

Our goal is to show that this is the same class of degrees previously considered and thus that Harrison orders are universal in the context of highness for isomorphism (i.e.\ restricting attention from all structures to Harrison orders gives no loss of generality), much as, e.g., computable Polish spaces are in the context of lowness for isomorphism \cite{fm-lowisom}. 

\begin{thm}\label{thm:highforisoandharrison}
The degrees which are (uniformly) high for isomorphism for Harrison orders are precisely the degrees which are (uniformly) high for isomorphism.
\end{thm}

We begin with the following lemma, recalling Notation \ref{TreeStar}.

\begin{lem}
$S\star T$ has an infinite path if and only if both $S$ and $T$ do.  Moreover, if $S$ and $T$ are computable, then a path through $S\star T$ uniformly computes paths through $S$ and $T$, and the join of a path through $S$ with a path through $T$ uniformly computes a path through $S\star T$.
\end{lem}

\begin{proof}
A sequence $((\sigma_0^i, \sigma_1^i))_{i \in \omega}$ is a path through $S\star T$ if and only if $(\sigma_0^i)_{i \in \omega}$ is a path through $S$ and $(\sigma_1^i)_{i \in \omega}$ is a path through $T$.
\end{proof}

\begin{proof}[Proof of Theorem~\ref{thm:highforisoandharrison}]
The backwards direction is immediate.  For the other direction, we will show that (uniform) highness for isomorphism for Harrison orders implies (uniform) highness for paths.  Specifically, given a computable tree $T$ generating a $\Pi^0_1$ class, we will uniformly exhibit a pair of computable linear orders such that if $[T]$ is nonempty, then the linear orders are both of order type $\omega_1^{ck}(1+\bQ)$, and we will give a uniform process to compute an element of $[T]$ from an isomorphism between the orders.  This will suffice to establish the result.

First, we fix a Harrison order $\+H$, i.e., a computable linear order of order type $\omega_1^{ck}(1+\bQ)$ with no descending hyperarithmetic sequence.  Let $\+L = \+H(1+\bQ)$.  Note that $\+L \cong \+H$: it is easy to see that $(1+\bQ)(1+\bQ)$ has a least element and is otherwise a countable dense linear order without endpoints, whence $(1+\bQ)(1+\bQ) \cong 1+\bQ$; and multiplication of linear orders is associative.  Furthermore, there is a computable embedding $\iota:\bQ \hookrightarrow \+L$: fix some $h \in \+H$, and define $\iota(q) = (h, q)$.  Let $S$ be the tree of finite descending sequences in $\+H$.

Now fix a computable tree $T$.
\begin{claim}
If $T$ has an infinite path, then $KB(S\star T)$, the Kleene-Brouwer ordering on $S\star T$, has order type $\omega_1^{ck}(1+\bQ) + 1$, where the final element is $\seq{}$, the root of $S\star T$.
\end{claim}

\begin{proof}[Proof of claim]
We fix $n \in \omega$ extendable to an infinite path through $T$ and note that $KB(S\star T) \cong \left(\sum_{i \in \omega} \+L_i\right)+1$, where $\+L_i$ is the Kleene-Brouwer ordering of the subtree extending the $i$th child of the root.  As none of these subtrees have a hyperarithmetic path, each $\+L_i$ is either well founded or has order type $\omega_1^{ck}(1+\bQ)+\alpha$ for some (possibly empty) well order $\alpha$.  To establish the claim, it suffices to show that there are infinitely many $i$ such that $\+L_i$ is ill founded.  However, each $i$ corresponding to a pair $(a, n)$ with $a$ in the ill-founded part of $\+H$ is extendible to a path through $S\star T$, and so such an $\+L_i$ is ill founded.  As there are infinitely many such $a$, the claim follows.
\end{proof}

Note that $S$ has the property that if $\sigma \in S$ is extendible to an infinite path, then $\sigma$ has infinitely many immediate children in $S$.  It follows that $S \star T$ has the same property.

Let $\+K$ be the order made from $KB(S\star T)$ by removing the root $\seq{}$.  Then $\+L$ and $\+K$ are our two orders of order type $\omega_1^{ck}(1+\bQ)$.  Observe that if $\tau \in S\star T$, then $\{ \rho \in S\star T: \rho \supset \tau\}$ is a computable, convex set in $\+K$.

Suppose now that $f: \+L \cong \+K$.  We recursively construct a sequence $(\tau_n)_{n \in \omega}$ of elements of $S\star T$ with $\tau_{n} \subset \tau_{n+1}$ and maintain the inductive assumption that $f\circ \iota$ maps a nontrivial interval of $\bQ$ into $\{ \rho \in S \star T : \rho \supset \tau_n\}$ (and thus that this interval of $K$ is ill founded, and so $\tau_n$ is extendible to a path through $S\star T$).  We begin with $\tau_0 = \seq{}$.  

Suppose we have defined $\tau_n$.  Then we let $(\alpha_m)_{m \in \omega}$ enumerate the immediate children of $\tau_n$.  Observe that $\bigsqcup_m \{ \rho \in S\star T : \rho \supseteq \alpha_m\}$ partitions $\{ \rho \in S \star T : \rho \supset \tau_n\}$, each component of the partition is convex in the ordering of $K$, and the components are arranged with order type $\omega$ (in the ordering of $K$).  So there must be rationals $q_0 < q_1$ and an $m$ with $f(\iota(q_0)), f(\iota(q_1)) \in \{ \rho \in S \star T : \rho \supseteq \alpha_m\}$.  We search for such an $m$ and define $\tau_{n+1} = \alpha_m$.  Now, we can observe that $\{\rho \in S\star T : \rho \supset \tau_{n+1}\}$ contains the open interval from $f(\iota(q_0))$ to $f(\iota(q_1))$, and so the inductive assumption is maintained.

Thus $f$ uniformly computes $(\tau_n)_{n \in\omega}$, which is a path through $S\star T$, and therefore it gives a path through $T$.  This completes the proof.
\end{proof}

\section{Scott analysis}\label{sec:scott}

As we consider our intuition for what might make a degree high for isomorphism, we naturally look to the possibility of back-and-forth arguments: the reason that the countable dense linear order without endpoints is computably categorical is that we can computably find the necessary extensions for a partial isomorphism.

A sweeping generalization of this condition is given by the so-called \emph{back-and-forth relations}.

\begin{dfn}
Let $\sA$ and $\sB$ be structures, and let $\bar{a}$ and $\bar{b}$ be finite tuples of the same length in $\sA$ and $\sB$, respectively.  We then define the following relations:
\begin{enumerate}
    \item We say that $(\sA,\bar{a}) \leq_1 (\sB,\bar{b})$ if and only if every finitary $\Sigma_1$ formula true of $\bar{b}$ in $\sB$ is true of $\bar{a}$ in $\sA$.
    \item For $\alpha>1$, we say that $(\sA,\bar{a}) \leq_\alpha (\sB,\bar{b})$ if and only if for every finite $\bar{d} \subseteq \sB$ and every $\beta < \alpha$ there is some finite $\bar{c} \subseteq \sA$ such that $(\sB,\bar{b}\bar{d}) \leq_\beta (\sA,\bar{a}\bar{c})$.
\end{enumerate}
\end{dfn}

In some cases---for instance, the countable dense linear order without endpoints---these are easy to compute.  In general, they are not.  For instance, among computable ordinals, these relations specify the place of each element of $\bar{a}$ in its copy of $\omega^\alpha$ for some appropriate $\alpha$.  These relations are discussed in detail in \cite{ashknight,Barwise1973}.

It is reasonable to hope that the degrees which are uniformly high for isomorphism would be exactly those degrees which can uniformly compute all of these relations.  Although this does not turn out to be true, we believe the concept is interesting and that the separation gives important information about both classes.

\begin{dfn}\label{ScComplete} We say that a degree $\mathbf{d}$ is \emph{Scott complete} if and only if there is an algorithm which, given indices for a computable structure $\sA$ and a Harrison order $\sH$, will give an index for a $\mathbf{d}$-computable function $\mathbf{b}:\left(\sA^{<\omega}\right)^2 \times \sH \to \{0,1\}$ with the following properties (we abbreviate $\mathbf{b}(\bar{a},\bar{b},x)=1$ by $\bar{a} \leq_x \bar{b}$):
\begin{enumerate}
\item If $\mathbf{1}$ is the least element of $\sH$, then $\bar{a} \leq_\mathbf{1} \bar{b}$ if and only if every finitary $\Sigma_1$ formula true of $\bar{b}$ is true of $\bar{a}$.
\item If $x \in \sH$ is not the least element, then $\bar{a} \leq_x \bar{b}$ if and only if for every finite $\bar{d} \subseteq \sA$ and every $y < x$ there is some finite $\bar{c} \subseteq \sA$ such that $\bar{b}\bar{d} \leq_y \bar{a}\bar{c}$.
\end{enumerate}
\end{dfn}

While we do not explicitly name the uniformity in this definition (having nothing to say at this point about a nonuniform variant), it is germane to note that the notion of Scott completeness belongs to the family of ``uniform" highness notions we explore in this paper.  Since we know that no degree can be both high for isomorphism and low for $\omega_1^{ck}$, the following result separates the two classes.

\begin{prop}\label{scLFO1CK}
There is a degree which is Scott complete and low for $\omega_1^{ck}$.
\end{prop}

\begin{proof}
The two conditions of Definition \ref{ScComplete} can be applied in the context of any linear order $\sL$.  Moreover, for any structure $\sA$ and any linear order $\sL$, there always exists a sequence $(\leq_x : x \in \sL)$ satisfying these conditions, since we can always take the standard back-and-forth relations for the greatest well-founded initial segment of $\sL$, and equality thereafter.

Now consider the set $S$ of all triples $(e,i,\mathbf{b})$ where $e$ is the index for a computable structure, $i$ is the index for a linear order, and $\mathbf{b}$ is a function satisfying the conditions of Definition \ref{ScComplete}.  The set of selectors $\eta:\mathbb{N}^2 \to 2^\omega$ such that for every $(e,i)$ we have $(e,i,\eta(e,i)) \in S$ is now a $\Sigma^1_1$ class.  By the Gandy Basis Theorem, this class has an element which is low for $\omega_1^{ck}$.
\end{proof}

The key difference is that a sequence $(\leq_x : x \in \+L)$ always exists for any structure $\+A$ and any linear order $\+L$.  In contrast, $\Sigma^1_1$ classes do not always have elements, and pairs of structures do not always have isomorphisms between them.

As the proof shows, Scott completeness follows from computing an element of a particular $\Sigma^1_1$ class, and thus by Proposition~\ref{prop:high_for_iso_vs_paths}, highness for isomorphism implies Scott completeness. We will find, however, that combining Scott completeness with an additional property gives a characterization of the degrees that are high for isomorphism (Theorem~\ref{prop:high_for_descending_and_Scott_complete_implies_high_for_isomorphism}).

\section{Computing descending sequences}\label{sec:descending}

Now we consider the task of finding descending sequences in linear orders.  The reader familiar with the theory of computable ordinals and the Kleene-Brouwer ordering may expect this to be the same as the task of finding paths through trees and thus the same as that of finding isomorphisms, but this turns out not to be the case.

\begin{dfn}
If $\+L$ is a linear order, a {\em tight descending sequence} in $L$ is an infinite descending sequence which is unbounded below in the ill-founded part of $\+L$, i.e., if $f$ is a tight descending sequence and $g$ is any descending sequence, then for every $n$ there is an $m$ with $f(m) <_{\+L} g(n)$.
\end{dfn}

\begin{dfn}
A degree $\* d$ is {\em high for (tight) descending sequences} if any computable ill-founded linear order $\+L$ has a $\* d$-computable (tight) descending sequence.

Degrees that are uniformly high for (tight) descending sequences are defined analogously.
\end{dfn}

As the descending sequences through a computable linear order form a $\Pi^0_1$ class, it follows by  Proposition~\ref{prop:high_for_iso_vs_paths} that (uniform) highness for isomorphism implies (uniform) highness for descending sequences.  We will show that this implication is strict.  As the set of indices for computable ill-founded linear orders is $\Sigma^1_1$-complete, Proposition~\ref{prop:jumps_of_high_are_powerful} holds for (uniform) highness for descending sequences in place of (uniform) highness for isomorphism, so the separation will not follow from a simple jump analysis.  In the nonuniform case, the separation turns out to be related to the question of exactly which degrees can be coded into the descending sequences of a linear order; we begin by introducing some notation.

\begin{notation}
For $\sigma$ a nonempty finite sequence, let $l(\sigma) = \sigma(|\sigma|-1)$ be the last element of $\sigma$.

For $\sbar = (\sigma_0, \dots, \sigma_{n-1})$ a collection of finite sequences, write $l(\sbar, i)$ for $l(\sigma_i)$.

For a linear order $\+L$, let $W(\+L)$ denote the (possibly empty) greatest well-founded initial segment.
\end{notation}

\begin{lem}\label{lem:ADS}
If $\+L$ is an infinite computable linear order, then either $\+L$ has a computable descending sequence or it has a $\mathbf{0}'$-computable ascending sequence.
\end{lem}

\begin{proof}
Suppose $\+L$ has no computable descending sequence.  Then $\+L$ must have a least element, as otherwise we can construct a computable descending sequence $(a_n)_{n \in \omega}$ by setting $a_0$ to be any element and then, given $a_n$, searching for an element less than $a_n$ to get $a_{n+1}$.  For an $x \in \+L$ that is not the greatest element of $\+L$, by applying the same argument to $\{ y \in \+L : x <_{\+L} y\}$, we see that $x$ must have an immediate successor.  Let $z$ be the least element of $\+L$, and for $x \in \+L$ not the greatest element, let $s(x)$ be its immediate successor.  Then $z, s(z), s(s(z)), \dots$ is a $\mathbf{0}'$-computable ascending sequence.
\end{proof}

It is not hard to construct an ill-founded computable linear order such that every descending sequence computes $\emptyset'$.  Similarly, one can construct $n$ ill-founded computable linear orders such that every degree computing a descending sequence through each of them computes $\emptyset^{(n)}$.  This turns out to be the best that is possible.

\begin{prop}\label{thm:descending_sequences_cone_avoidance}
If $\+L_0, \dots, \+L_{n-1}$ are computable ill-founded linear orders and $X \not \in \Delta^0_{n+1}$, then there are descending sequences $f_0, \dots, f_{n-1}$ through $\+L_0, \dots, \+L_{n-1}$, respectively, such that $X \not \le_T f_0 \oplus \dots \oplus f_{n-1}$.
\end{prop}

\begin{proof}
The proof proceeds by induction on $n$.

First, suppose there is $k > 0$ such that at least $k$ of the linear orders have a descending $\mathbf{0}^{(k)}$-computable sequence.  If $k = n$, the theorem follows.  Otherwise, without loss of generality, $\+L_{n-k}, \dots, \+L_{n-1}$ have descending $\mathbf{0}^{(k)}$-computable sequences $f_{n-k}, \dots, f_{n-1}$.  Then, as $X \not \in \Delta^0_{n-k+1}(\mathbf{0}^{(k)})$, by the inductive hypothesis relative to $\mathbf{0}^{(k)}$ there are descending sequences $f_0, \dots, f_{n-k-1}$ through $\+L_0, \dots, \+L_{n-k-1}$ such that $X \not \le_T f_0\oplus \dots \oplus f_{n-k-1} \oplus \mathbf{0}^{(k)}$, and thus $X \not \le_T f_0\oplus \dots \oplus f_{n-1}$.

Now suppose there is no such $k$.  We use forcing in which a condition is a tuple $(\sigma_0, \dots, \sigma_{n-1})$ such that for $i < n$, $\sigma_i$ is a finite descending sequence through $\+L_i$, and if $\sigma_i$ is nonempty, $l(\sigma_i) \not \in W(\+L_i)$. We say that  $(\tau_0, \dots, \tau_{n-1})$ extends $(\sigma_0,\dots, \sigma_{n-1})$ if each $\tau_i$ extends $\sigma_i$.  Note that in general the set of conditions is properly $\Sigma^1_1$.  Clearly, a sufficiently generic filter  for this notion of forcing gives a tuple $(f_0,\dots, f_{n-1})$ such that $f_i$ is an infinite descending sequence through $\+L_i$ for $i < n$.  We will show that a sufficiently generic filter gives $f_0\oplus \dots \oplus f_{n-1} \not \ge_T X$.

For $\Phi$ a functional and $\sbar = (\sigma_0, \dots, \sigma_{n-1})$, we will write $\Phi^{\sbar}$ for $\Phi^{\sigma_0 \oplus \dots \oplus \sigma_{n-1}}$.

Suppose $\sbar$ is a condition and $\Phi$ is a Turing functional.  We must show that $\sbar$ has an extension $\tbar$ with $\tbar \Vdash [\Phi^{f_0 \oplus \dots \oplus f_{n-1}} \neq X]$.  If there is an extension $\tbar$ and an $m$ with $\Phi^{\tbar}(m)\converge \neq X(m)$, then we are done, and if there is an $m$ and a $\tbar \preceq \sbar$ such that for all $\rbar \preceq \tbar$, $\Phi^{\rbar}(m)\diverge$, then we are also done.  Now let us suppose towards a contradiction that neither of these hold.

Let $D = \{ (\alpha_0, \dots, \alpha_{n-1}) : \alpha_i \supseteq \sigma_i\}$; $D$ is certainly computable.  We will use $\abar, \bbar, \gbar, \dbar$ for elements of $D$, in contrast with $\sbar, \tbar, \rbar$ for conditions, and as before, we will write $\Phi^{\abar}$ for $\Phi^{\alpha_0\oplus\dots \oplus \alpha_{n-1}}$.

For $\abar, \bbar \in D$, we write $\abar \tleq \bbar$ if $l(\abar, i) \le_{L_i} l(\bbar, i)$ for all $i < n$; we further write $\abar \bot \bbar$ if $\abar \not \tleq \bbar$, $\bbar \not \tleq \abar$, and, for each $i < n$, $l(\abar,i) \neq l(\bbar,i)$.  We define extension on $D$ in the natural fashion.

We construct uniformly $\Sigma^0_1(\mathbf{0'})$ sets $(A_k)_{k \in \omega}$, each a subset of $D$, as follows.
\begin{itemize}
\item $A_0 = \{\sbar\}$.
\item If we see $\abar \in D$ satisfying the following:
\begin{itemize}
\item $\abar$ extends some element of $A_k$;
\item $\Phi^{\abar}(k)\converge$;
\item There are no $\bbar, \bbar' \in D$ and $m \in \omega$ with $\abar \tleq \bbar$, $\abar \tleq \bbar'$ and $\Phi^{\bbar}(m)\converge \neq \Phi^{\bbar'}(m)\converge$; and
\item $\abar \bot \gbar$ for every $\gbar$ already enumerated into $A_{k+1}$;
\end{itemize}
then we enumerate $\abar$ into $A_{k+1}$.
\end{itemize}

\begin{claim}
For every $k$ and $\abar \in A_k$, for some $i < n$, $l(\abar, i) \not \in W(\+L_i)$.
\end{claim}

\begin{proof}
Suppose not and fix an $\abar$ for which this fails.  Then for every condition $\rbar$, $\abar \tleq \rbar$.  Observe that 
\[
X(m) = b \iff \exists \bbar \in D\, [(\abar \tleq \bbar) \wedge \Phi^{\bbar}(m) = b].
\]
This holds because if $X(m) = b$, then by assumption on $\sbar$, a condition $\rbar \preceq \sigma$ with $\Phi^{\rbar}(m) = b$ exists, and this is our witnessing $\bbar$.  Then the existence of the $\rbar$ for $X(m)$ and $\abar \in A_k$ implies there cannot exist a witnessing $\bbar'$ for $b \neq X(m)$.

Thus $X$ is computable, contrary to assumption.
\end{proof}

\begin{claim}
If $A_m$ is finite for every $m \le k$, then $A_k$ contains a condition (and in particular, $A_k \neq \emptyset$).
\end{claim}

\begin{proof}
We proceed by induction on $k$. The result for $k = 0$ is immediate.

Now consider the inductive case $k+1$.  By the inductive hypothesis, we can fix a condition $\abar = (\alpha_0, \dots, \alpha_{n-1}) \in A_k$.  As $|A_{k+1}|$ is finite, we fix $b_i <_{\+L_i} l(\abar, i)$ for each $i < n$ such that for every $\gbar \in A_{k+1}$, if $l(\gbar,i) \not \in W(L_i)$, then $b_i <_{\+L_i} l(\gbar, i)$.  Let $\abar' = (\alpha_0\cat b_0, \dots, \alpha_{n-1}\cat b_{n-1})$.  We can see that $\abar'$ is a condition.  By the previous claim, $\gbar \not \tleq \abar$ for any $\gbar \in A_{k+1}$.

By our assumption on $\sbar$, there is a condition $\tbar \preceq \abar'$ with $\Phi^{\tbar}(n)\converge$.  There cannot be $\bbar, \bbar' \in D$ and $m \in \omega$ blocking $\tbar$'s enumeration into $A_{k+1}$, as both would be conditions extending $\sbar$, and one of them would force $\Phi^{f_0 \oplus \dots \oplus f_{n-1}}(m)\converge \neq X(m)$ contrary to assumption on $\sbar$.

Eventually we will locate $\tbar$.  As $l(\tbar, i) \le_{\+L_i} b_i$ for $i < n$, it cannot be that there is a $\gbar \in A_{k+1}$ such that $\gbar \tleq \tbar$.  Thus if there is $\gbar \in A_{k+1}$ such that $\neg(\gbar \bot \tbar)$, it must be that $\tbar \tleq \gbar$, and thus $\gbar$ is a condition, so $\gbar$ is the desired element of $A_{k+1}$.  Otherwise, $\tbar$ will be enumerated into $A_{k+1}$ and $\tbar$ is the desired element of $A_{k+1}$.
\end{proof}

Note that if $n = 1$, then the $\abar \bot \gbar$ condition on $A_{k+1}$ implies each $A_{k+1}$ is a singleton.  It follows that, contrary to assumption, $\+L_0$ has a $\mathbf{0}'$-computable descending sequence.  Henceforth, we assume $n > 1$.

\begin{claim}
There is a $k$ such that $A_k$ is infinite.
\end{claim}

\begin{proof}
Suppose not and consider the tree $T$ of finite sequences $\pi$ such that for all $k < |\pi|$, $\pi(k) \in A_k$, and for $k < |\pi|-1$, the sequence $\pi(k+1)$ extends $\pi(k)$.  Note that a path through $T$ gives a tuple $(g_0, \dots, g_{n-1})$ such that $g_i$ is a descending sequence through $\+L_i$ extending $\sigma_i$ for $i < n$ and such that $\Phi^{g_0\oplus \dots \oplus g_{n-1}} = X$ is total.  Also, as the $A_k$s are finite and uniformly $\Sigma^0_2$, the class $[T]$ is effectively compact relative to $\mathbf{0}''$.  Thus $X \in \Delta^0_3$ contrary to assumption.
\end{proof}

Now we can fix $k$ such that $A_k$ is infinite.  For $i < n$, let $B_i = \{ l(\abar,i) : \abar \in A_k\}$.  Note that for $\abar, \abar' \in A_k$, since $\abar \bot \abar'$, we have that $l(\abar,i) \neq l(\abar',i)$ for all $i < n$.  Thus each $B_i$ is infinite.  Furthermore, each $B_i$ is $\Sigma^0_2$.

\begin{claim}
Each of the $\+L_i$s has a $\mathbf{0}^{(n)}$-computable descending sequence, contrary to assumption.
\end{claim}

\begin{proof}
Fix $i$ and define a sequence $c_0, \dots, c_{n-1}$ enumerating $\{0, \dots, n-1\}$ as follows:
\begin{itemize}
\item $c_{n-1} = i$;
\item Having defined $c_{j+1}, \dots, c_{n-1}$, by assumption there is
\[
r \in \{0, \dots, n-1\} \setminus \{c_{j+1}, \dots, c_{n-1}\}
\]
such that $L_r$ has no $\mathbf{0}^{(j+1)}$-computable descending sequence.  Let $c_j$ be such an $r$.
\end{itemize}
We now recursively build a sequence of injective functions $h_j: \omega \to A_k$ and sets $C_j \subseteq B_{c_j}$ for $j < n-1$ such that $h_j$ is $\mathbf{0}^{(j+2)}$-computable and $C_j$ is $\Sigma^0_{j+2}$ and infinite.  We construct them as follows:
\begin{itemize}
\item $C_{0} = B_{c_0}$;
\item For $j < n-1$, given $C_j \subseteq B_{c_j} \subseteq \+L_{c_j}$, give $C_j$ the ordering induced by $\+L_{c_j}$.  Then $C_j$ is $\mathbf{0}^{(j+1)}$-computably isomorphic to a $\mathbf{0}^{(j+1)}$-computable linear order.  By Lemma~\ref{lem:ADS} relative to $\mathbf{0}^{(j+1)}$ and the fact that $\+L_{c_j}$ has no $\mathbf{0}^{(j+1)}$-computable descending sequence, there is some $\mathbf{0}^{(j+2)}$-computable ascending sequence through $C_j$.  Each element of this sequence comes from some unique $\abar \in A_k$, so let $h_j$ list these elements of $\abar$ in the order given by the ascending sequence.  That is, the ascending sequence is
\[
l(h_j(0),c_j) <_{\+L_{c_j}} l(h_j(1),c_j) <_{\+L_{c_j}} \dots.
\]
\item For $j < n-2$, given $h_j$, let
\[
C_{j+1} = \{ l(h_j(m),c_{j+1}) : m \in \omega\}.
\]
As $h_j$ is injective and distinct elements of $A_k$ give distinct elements of $B_{c_{j+1}}$, $C_{j+1}$ is infinite.
\end{itemize}
We observe that for $j < n-1$, $\text{range}(h_{j+1}) \subseteq \text{range}(h_j)$.  Thus, there is some injective $g:\omega \to \omega$ such that $h_{j+1} = h_j \circ g$.  By thinning the ascending sequence which gives rise to $h_{j+1}$, we may assume that $g$ is ascending.  This can be done without affecting the complexity of $h_{j+1}$.

It follows that for all $m_0 < m_1$ and $j < n-1$, we have  \[l(h_{n-2}(m_0),c_j) <_{\+L_{c_j}} l(h_{n-2}(m_1),c_j).\]  As $h_{n-2}(m_0)\bot h_{n-2}(m_1)$, it must be true that \[l(h_{n-2}(m_0),i) >_{\+L_i} l(h_{n-2}(m_1),i).\]  As $h_{n-2}$ is $\mathbf{0}^{(n)}$-computable, this is a $\mathbf{0}^{(n)}$-computable descending sequence in $\+L_i$.
\end{proof}
This completes the proof of Proposition~\ref{thm:descending_sequences_cone_avoidance}.
\end{proof}

\begin{lem}\label{lem:no_tight_means_no_loose}
If $\+L$ is a computable ill-founded linear order with no arithmetic tight descending sequences, then for every $n$ there is an $a \in \+L\setminus W(\+L)$ such that $\{z \in \+L : z <_{\+L} a\}$ has no $\Delta^0_n$ descending sequences.
\end{lem}

\begin{proof}
Suppose not, and fix an $n$ for which this fails.  Then \[\+L\setminus W(\+L) = \{ a : \text{$a$ bounds a $\Delta^0_n$ descending sequence}\}\] is arithmetic, and as it has no least element, there is an arithmetic tight descending sequence in $\+L$.
\end{proof}

\begin{por}
If $\+L_0, \dots, \+L_n$ are ill-founded linear orders without arithmetic tight descending sequences and $X$ is not arithmetic, then a sufficiently generic filter for the forcing notion described in the proof of Proposition~\ref{thm:descending_sequences_cone_avoidance} will give tight descending sequences through the $\+L_i$ such that the join of these sequences does not compute $X$.
\end{por}

\begin{proof}
Clearly a sufficiently generic filter gives a tight descending sequence.  Given a condition $\sbar$ and a functional $\Phi$, we may first extend $\sbar$ such that each $l(\sbar, i)$ bounds no $\Delta^0_{n+1}$ descending sequence in $\+L_i$.  By the arguments given, $\sbar$ can be extended to a $\tbar$ such that $\tbar \Vdash \Phi^{f_0\oplus \dots\oplus f_{n-1}} \neq X$.
\end{proof}

Proposition~\ref{thm:descending_sequences_cone_avoidance} demonstrates the limit of what can be coded into the descending sequences of finitely many linear orders.  We consider the result to be of independent interest, but the proof also serves as an opportunity to develop the tools we will need to show a similar result for degrees that are high for descending sequences in Theorem~\ref{thm:high_for_descending_sequences_avoids_nonarithmetical}.

\subsection{Degrees that are high for descending sequences}

\begin{thm}\label{thm:high_for_descending_sequences_avoids_nonarithmetical}
If $X$ is not arithmetical, there is a degree that is high for tight descending sequences and does not compute $X$.
\end{thm}

\begin{proof}
Let $\{\+L_e : e \in \omega\}$ be a listing of the computable ill-founded linear orders without arithmetic tight descending sequences.  This cannot be an effective listing, but its true complexity will not be relevant.

We construct our degree via forcing. Here, a condition is a pair $(\sigma_0, \sigma_1)$ satisfying the following:
\begin{itemize}
\item $\sigma_0: \omega^2 \to \omega$ and $\sigma_1: \omega \to \omega \times \{0,1\}$ are finite partial functions; and
\item If $\sigma_1(n) = (e,0)$, then $\tau = \lambda x.\sigma_0(n,x)$ is a finite descending sequence in $\mathcal{L}_e$ that is extendible to an infinite descending sequence.
\end{itemize}

A condition $\sbar' = (\sigma_0', \sigma_1')$ extends $\sbar = (\sigma_0, \sigma_1)$ if the following hold:
\begin{itemize}
\item $\sigma_0'$ extends $\sigma_0$ and $\sigma_1'$ extends $\sigma_1$;
\item If $\sigma_1(n) = (m,1)$, then for $\tau = \lambda x.\sigma_0(n, x)$ and $\tau' = \lambda x.\sigma_0'(n,x)$, we have $\tau'(x) = \emptyset^{(m)}(x)$ for all $x$ such that $|\tau| \le x < |\tau'|$.
\end{itemize}
Given a filter $F$ for this notion of forcing, we let $f = \bigcup\limits_{(\sigma_0, \sigma_1) \in F} \sigma_0$.  If $F$ is sufficiently generic, then for every $m$ there is an $n$ and a $(\sigma_0, \sigma_1) \in F$ with $\sigma_1(n) = (m,1)$, and for this $n$, $\lambda x.f(n, x) =^* \emptyset^{(m)}$. Thus, $f$ will bound every arithmetic set.  Furthermore, if $F$ is sufficiently generic, then for every $e$ there is an $n$ and a  $(\sigma_0, \sigma_1) \in F$ with $\sigma_1(n) = (e, 0)$.  For this $n$, $\lambda x.f(n,x)$ will be a tight descending sequence through $\+L_e$.  It follows that $f$ is high for descending sequences.

Suppose we have a condition $\sbar$ and a Turing functional $\Phi$.  We wish to argue that $\sbar$ can be extended to force $\Phi^f \neq X$.  Let $m = \max\left\{ m' : \exists n < |\sigma_1|\ \left[\sigma_1(n) = (m', 1)\right]\right\}$.  Let $E = \left\{ e : \exists n < |\sigma_1|\ \left[\sigma_1(n) = (e,0)\right]\right\}$.  Then any $\tbar$ extending $\sbar$ is committed to writing sets $=^* \emptyset^{(m')}$ for some $m' \le m$ on some finitely many columns of $\tau_0$ and to writing descending sequences through $\+L_e$ for $e \in E$ on other columns of $\tau_0$, but is unconstrained on the remaining columns of $\tau_0$.

Let $n = |E|$.  By Lemma~\ref{lem:no_tight_means_no_loose}, we may first extend the descending sequences coded in $\sigma_1$ such that none of them can be extended to $\Delta^0_{n+m+1}$ descending sequences.  Now, working relative to $\emptyset^{(m)}$, the main argument in the proof of Theorem~\ref{thm:descending_sequences_cone_avoidance} shows that we can extend to such a $\tbar$ (we are in the ``no such $k$'' case in the proof of that theorem).

Thus for a sufficiently generic filter, $\Phi^f \neq X$, and the degree of $f$ is our desired degree.
\end{proof}

This stands in contrast with the high for isomorphism degrees, which compute all $\Delta^1_1$ sets (Proposition~\ref{prop:high_for_iso_above_hyp}).

\begin{cor}
There is a degree which is high for tight descending sequences but not high for isomorphism.
\end{cor}
Of course, the above result is heavily nonuniform.  We now turn our attention to the uniform case, where we will show that we can separate the uniform and non-uniform notions, but not merely by looking at computation of hyperarithmetical degrees.  Toward this end, we begin with the following technical result.

\begin{lem}\label{lem:uniformlyhighdescendingandbounding}
Fix $(U_n)_{n \in\omega}$ an effective listing of $\Sigma^1_1$ subsets of $\omega$.  A degree $\* d$ is uniformly high for descending sequences if and only if there is a $\* d$-computable partial function $f: \omega \to \omega$ such that if $U_n$ is nonempty, then $f(n)\converge \ge \min U_n$.
\end{lem}

\begin{proof}
Suppose $\*d$ is uniformly high for descending sequences.  We construct $f$ as follows.  On input $n$, for each $m \in \omega$, we let $T_m$ be a computable tree having a path if and only if $m \in U_n$.  We take $\+L_m$ to be the Kleene-Brouwer ordering of $T_m$, so $U_n$ is nonempty if and only if $\+L = \+L_0 + \+L_1 + \dots$ is ill founded.  We ask $\* d$ to provide us with a descending sequence in $\+L$ as in the definition of uniformly high for descending sequences.  If $\* d$ produces an output, we define $f(n) = m$, where $m$ is such that the first element of the putative descending sequence occurs in $\+L_m$.

If $U_n$ is truly nonempty, then $\* d$ will produce output, and the first element of this output will bound a descending sequence, and so some $\+L_{m'}$ with $m' \le m$ will be ill founded and thus $m' \in U_n$.

\smallskip

Now suppose $\* d$ computes such an $f$.  Given an ill-founded linear order $\+L$, we construct a descending sequence as follows.  First, we fix $n_0$ such that $U_{n_0} = \+L\setminus W(\+L)$.  This is nonempty by assumption, and so $f(n_0)\converge$ with $\{0, \dots, f(n_0)\} \cap U_{n_0} \neq \emptyset$.  Let $a_0$ be the $\+L$-rightmost element of $\{0, \dots, f(n_0)\}$; this gives us $a_0 \in U_{n_0}$.

Given $a_i \in U_{n_0}$, we fix $n_{i+1}$ such that $U_{n_{i+1}} = \{x \in U_{n_0} : x <_{\+L} a_i\}$.  This is nonempty by assumption, so $f(n_{i+1})\converge$ with $\{0, \dots, f(n_{i+1})\} \cap U_{n_{i+1}} \neq \emptyset$.  Let $a_{i+1}$ be the $\+L$-rightmost element amongst those $\{0, \dots, f(n_{i+1})\}$ which are to the left of $a_i$; this gives us $a_{i+1} \in U_{n_{i+1}}$.

Now the sequence $(a_i)_{i \in \omega}$ is a descending sequence in $\+L$ and is uniformly computable from $f$ and an index for $\+L$.
\end{proof}

\begin{thm}\label{prop:pa_over_unif_desc_seq_is_unif_iso}
If $\* d$ is uniformly high for descending sequences and $\* b$ is PA over $\* d$, then $\* b$ is uniformly high for isomorphism.
\end{thm}

\begin{proof}
Fix a computable ill-founded tree $T$.  We show how to uniformly prune $T$ to a $\* d$-computable, finitely branching tree (with a $\* d$-computable bound on the branches).  Then, as $\* b$ is PA over $\* d$, it will uniformly compute a path through said tree. 

 We fix $f$ as in Lemma~$\ref{lem:uniformlyhighdescendingandbounding}$ and proceed recursively.  First, we define $n_0$ such that $U_{n_0} = \{ m \in \omega : \text{$\seq{m}$ is extendible in $T$}\}$.  Given $n_0, \dots, n_{k-1}$, we then define \[T_k = \{\sigma \in 2^{k}: (\forall i < k)\, \sigma(i) \le f(n_i)\}\] and, then,  $n_k$ such that \[U_{n_k} = \{ m \in \omega : (\exists \sigma \in T_k)\, \text{$\sigma\cat m$ is extendible in $T$}\}.\] Inductively, using the properties of $f$, each $T_k$ contains a string extendible in $T$.  Then the tree $\{ \sigma\in T : (\forall i < |\sigma|)\, \sigma(i) \le f(n_i)\}$ is our pruned tree.\end{proof}

\begin{cor}
If $\* d$ is uniformly high for descending sequences, then $\* d$ computes every $\Delta^1_1$ set.
\end{cor}

\begin{proof}
Fix $X \in \Delta^1_1$, and suppose $\* d$ does not compute $X$ for a contradiction.  By the cone avoidance basis theorem~\cite{GaKrTa60}, there is a $\* b$ that is PA over $\* d$ which does not compute $X$.  But then $\* b$ is uniformly high for isomorphism, contradicting Proposition~\ref{prop:high_for_iso_above_hyp}.
\end{proof}

Thus, we cannot separate uniform highness for descending sequences from highness for isomorphism simply by looking at the sorts of degrees they compute as we did for the nonuniform case.  Nevertheless, they are separate classes of degrees.  The remainder of this section is devoted to separating these classes, though we will need to develop a number of tools first.

\subsection{Computing jump hierarchies}\label{sec:jumps}

Given a computable linear order $\sL$, a \emph{jump hierarchy} on $\sL$ is a function $h:\sL \to 2^\omega$ such that for all $x \in \sL$, $h(x) = \bigoplus_{y <_{\+L} x} h(y)'$.  For every computable well order, of course, an initial segment of the standard hyperarithmetical hierarchy constitutes a jump hierarchy \cite{Sacks1990}.  There are, however, linear orders which do not admit a jump hierarchy, and this property is not invariant under isomorphism.  In particular, there are Harrison orders with jump hierarchies and those without.

The terminology for jump hierarchies varies in the literature.  A jump hierarchy is sometimes called a \emph{jump structure} or a \emph{Turing jump hierarchy}.

In any case, computing a jump hierarchy on a long linear order is a strong standard for a degree.  For instance, no hyperarithmetical degree can compute a jump hierarchy on any Harrison order.  This gives rise to the following definition.

\begin{dfn} We say that a degree $\mathbf{d}$ is \emph{jump complete} if and only if there is an algorithm which, given an index for a computable Harrison order which admits a jump hierarchy, will compute such a hierarchy from oracle $\mathbf{d}$.
\end{dfn}

\begin{prop} Every degree which is uniformly high for isomorphism is jump complete.\end{prop}

\begin{proof}
Given a Harrison order $\sH$, the set of jump hierarchies on $\sH$ is a $\Sigma^1_1$ class.  A degree that is uniformly high for isomorphism can compute an element of such a class as long as it is nonempty.
\end{proof}

\begin{prop}\label{scNOTjc} There is a degree that is Scott complete but not jump complete.\end{prop}

\begin{proof}
By Proposition~\ref{scLFO1CK}, there is a degree that is Scott complete and low for $\omega_1^{ck}$. On the other hand, Harrington (unpublished; see~\cite{JLGip}) showed that the set of computable linear orders admitting jump hierarchies is $\Sigma^1_1$-complete, and so the same argument as in Proposition~\ref{prop:jumps_of_high_are_powerful}(2) shows that if $\*d$ is jump complete, then $\*d'' \ge_T \+ O$.
\end{proof}

Jump hierarchies are interesting in their own right, but they will also be an important tool to separating uniform highness for descending sequences from highness for isomorphism.  In particular, in Corollary~\ref{cor:separating_uniformly_high_desc_from_high_iso} we will construct a degree $\*d$ which is uniformly high for descending sequences and a linear order $\+L$ admitting jump hierarchies, and such that $\*d$ does not compute any jump hierarchy on $\+L$.  This gives the following.

\begin{prop} There is a degree that is uniformly high for descending sequences but not jump complete.\end{prop}

\subsection{$\Pi^1_1$-tracing of majorized classes}

We continue developing the tools necessary for our intended separation result. Here, we consider traceability.

\begin{dfn}
For $\sigma, \tau \in \omega^{\le \omega}$, we say that $\sigma$ {\em majorizes} $\tau$ if for every $x \in \dom(\sigma)\cap \dom(\tau)$, $\sigma(x) \ge \tau(x)$.
\end{dfn}

The following is an effectivization of results surrounding Laver forcing in set theory (see Bartoszynski~\cite{bartjud95} \S 7.3.D).

\begin{dfn}
An {\em order} is a nondecreasing function $h: \omega \to \omega$ such that $h(0) > 0$ and $\lim_s h(s) = \infty$.

For $h$ an order and $f \in \omega^\omega$, an {\em $h$-trace of $f$} is a sequence $(V_n)_{n \in \omega}$ such that for all $n$, $|V_n| \le h(n)$ and $f\uhr{n} \in V_n$.

An $h$-trace $(V_n)_{n \in \omega}$ is $\Pi^1_1$ if the $V_n$s are uniformly $\Pi^1_1$.
\end{dfn}

Our goal is the following:
\begin{lem}\label{lem:tracing}
If $\+P \subset \omega^\omega$ has compact closure, $\Phi$ is a Turing functional, and $f \in \omega^\omega$ is such that for every $g$ majorizing $f$, $\Phi^g \in \+P$, then for every computable order $h$, there is an element of $\+P$ with a $\Pi^1_1$ $h$-trace.
\end{lem}

The proof of this lemma will require a number of ingredients. We begin with the following definitions.

\begin{dfn} \hfill \\
\begin{enumerate}
    \item A {\em Laver tree with stem $\tau$} is a tree $T \subseteq \omega^{<\omega}$ with $\tau \in T$ such that for every $\sigma \in T$ with $\sigma \not \subset \tau$, there are infinitely many $n \in \omega$ with $\sigma\cat n \in T$.

\item A {\em stemless Laver tree} is a Laver tree with stem $\seq{}$.

\item A set $A \subseteq T$ {\em covers} $T$ if every path through $T$ meets an element of $A$.
\end{enumerate}
\end{dfn}

\begin{notation}
If $T$ is a Laver tree with stem $\tau$ and $\sigma \supseteq \tau$ is on $T$, then
\[
T^\sigma = \{ \rho \in T: \rho \supseteq \sigma\} \cup \{ \rho : \rho \subseteq \sigma \}.
\]
\end{notation}
Observe that $T^\sigma$ is a Laver tree with stem $\sigma$, and also that if $T$ is a stemless Laver tree, then for every $f \in \omega^\omega$ there is a $g \in [T]$ majorizing $f$.
Our intention is to build a $\Pi^1_1$ sequence $(V_n)_{n \in \omega}$ and a stemless Laver tree $T$ such that for every $g \in [T]$ with $\Phi^g$ total, $(V_n)_{n \in \omega}$ is an $h$-trace for $\Phi^g$.

\begin{dfn}
Suppose $T$ is a Laver tree with stem $\tau$ and $A \subseteq T$ is upwards closed in $T$.  Let $\+L$ be a linear order with distinguished least and greatest elements, respectively 0 and $\infty$, and a successor function $+1$ (with $\infty+1 = \infty$).  An {\em $(A, \+L, T)$ Laver ranking} is a function $r: \{ \sigma \in T: \tau \subseteq \sigma\} \to \+L$ satisfying the following:
\begin{itemize}
\item $r(\sigma) = 0$ if and only if $\sigma \in A$; and
\item For $\sigma \not \in A$, $r(\sigma) = \liminf_{\sigma\cat n \in T} [r(\sigma\cat n)+1]$.
\end{itemize}
\end{dfn}

Observe that for $\+L$ countable, the collection $LR(A,\+L,T)$ of $(A,\+L,T)$ Laver rankings is an arithmetic class relative to $A, \+L$ and $T$.

\begin{lem}
For any Laver tree $T$ with stem $\tau$ and $A \subseteq T$ upwards closed, an $(A, \omega_1\cup \{\infty\},T)$ Laver ranking exists.
\end{lem}

\begin{proof}
We construct an increasing sequence of partial functions $(r_{\alpha})_{\alpha < \omega_1}$.  We begin by defining $r_0(\sigma) = 0$ for $\sigma \in A$ and $r_0(\sigma)\diverge$ for $\sigma \not \in A$.

Having defined $r_\beta$ for $\beta < \alpha$, we let $r_{<\alpha} = \bigcup_{\beta < \alpha}r_\beta$.  If $r_{<\alpha}(\sigma)\converge$, then we define $r_\alpha(\sigma) = r_{<\alpha}(\sigma)$.  If $r_{<\alpha}(\sigma)\diverge$ but there are infinitely many $n$ such that $\sigma\cat n \in T$ and $r_{<\alpha}(\sigma\cat n)\converge$, then we define $r_\alpha(\sigma) = \alpha$.  Otherwise, we leave $r_\alpha(\sigma)$ undefined.

By cardinality considerations, there is $\lambda <\omega_1$ such that $r_\lambda = r_{\lambda+1}$.  We define $r(\sigma) = r_\lambda(\sigma)$ for $\sigma \in \dom(r_\lambda)$ and $r(\sigma) = \infty$ for $\sigma \not \in \dom(r_\lambda)$.

Clearly $r$ is as required.
\end{proof}
Let $\dot{r}_{A,T}$ be the $(A, \omega_1\cup\{\infty\}, T)$ Laver ranking just constructed.  Lemmas~\ref{lem:no_r_smaller} and~\ref{lem:wf_ranks_same} will imply that it is unique.  Observe that for sufficiently large countable $\lambda$, $\dot{r}_{A,T}$ is an $(A, \lambda \cup \{\infty\}, T)$ Laver ranking as well.  This Laver ranking tells us about the existence of Laver trees.

\begin{dfn}
For $T$ a Laver tree with stem $\tau$ and $A \subseteq T$ upwards closed, define $T_A$ recursively:
\begin{itemize}
\item If $\sigma \subseteq \tau$, then $\sigma \in T_A$;
\item If $\sigma \in T_A$ with $\sigma \supseteq \tau$ and $\sigma\cat n \in T$ with $\dot{r}_{A,T}(\sigma\cat n) < \dot{r}_{A,T}(\sigma) < \infty$, then $\sigma\cat n \in T_A$;
\item If $\sigma \in T_A$ with $\sigma \supseteq \tau$ and $\sigma \in A$, then for every $\sigma\cat n \in T$,  $\sigma\cat n \in T_A$; and
\item If $\sigma \in T_A$ with $\sigma \supseteq \tau$ and $\sigma\cat n \in T$ with $\dot{r}_{A,T}(\sigma\cat n) = \dot{r}_{A,T}(\sigma) = \infty$, then $\sigma\cat n \in T_A$.
\end{itemize}
\end{dfn}

We now present some basic facts about these $T_A$s.

\begin{lem}
If $T$ is a Laver tree with stem $\tau$ and $\dot{r}_{A,T}(\tau) = \infty$, then $T_A$ is a Laver tree and no element of $[T_A]$ meets $A$.

If $T$ is a Laver tree with stem $\tau$ and $\dot{r}_{A,T}(\tau) < \infty$, then $T_A$ is a Laver tree covered by $A \cap T_A$.

In either case, for all $\sigma \in T_A$ with $\sigma \supseteq \tau$, $\dot{r}_{A, T}(\sigma) = \dot{r}_{A, T_A}(\sigma)$.
\end{lem}

\begin{proof}
Suppose $\dot{r}_{A,T}(\tau) = \infty$.  By induction on $|\sigma|$, if $\tau \subseteq \sigma \in T_A$, then $\dot{r}_{A,T}(\sigma) = \infty$.  As $T_A \cap A = \emptyset$, all $\sigma \supseteq \tau$ on $T_A$ have $\dot{r}_{A, T_A}(\sigma) = \infty$.

Now suppose $\dot{r}_{A,T}(\tau) < \infty$.  Then for any $g \in [T_A]$, $\dot{r}_{A,T}$ is decreasing along $g$ after $\tau$ until it reaches $0$.  For $\sigma \in T_A$ with $\sigma \supseteq \tau$, by induction on $\dot{r}_{A, T}(\sigma)$, $\dot{r}_{A, T}(\sigma) = \dot{r}_{A, T_A}(\sigma)$.
\end{proof}

\begin{lem}\label{lem:ranks_downward_closed}
For any $A \subseteq T$, $\text{range}\,(\dot{r}_{A,T})\setminus\{\infty\}$ is closed downwards.
\end{lem}

\begin{proof}
Fix $\beta \le \alpha$ and $\sigma$ with $\dot{r}_{A,T}(\sigma) = \alpha$.  We must show that there is a $\tau$ with $\dot{r}_{A,T}(\tau) = \beta$.  We construct a sequence $\sigma = \sigma_0 \subset \sigma_1 \subset \dots$ with $\beta \le \dot{r}_{A,T}(\sigma_{i+1}) < \dot{r}_{A,T}(\sigma_i)$.

Suppose we have defined $\sigma_i$.  If $\dot{r}_{A,T}(\sigma_i) = \beta$, we are done.  Otherwise, there is an $n$ with $\beta \le \dot{r}_{A,T}(\sigma_i\cat n) < \dot{r}_{A,T}(\sigma_i)$.  Define $\sigma_{i+1} = \sigma_i\cat n$ for such an $n$.

By the well-foundedness of the ordinals, this sequence must terminate in a $\sigma_i$ with $\dot{r}_{A,T}(\sigma_i) = \beta$.
\end{proof}

We will sometimes abuse notation by writing $\dot{r}_{A, T}$ when $A \not \subseteq T$, in which case this should be understood as $\dot{r}_{A\cap T, T}$.  This will not be a problem because of the following two lemmas.

\begin{lem}\label{lem:smaller_set_bigger_laver_rank}
Let $T$ be a Laver tree with stem $\tau$, and let $A \subseteq B \subseteq T$.  Then for all $\sigma \supseteq \tau$ with $\sigma \in T$, $\dot{r}_{B, T}(\sigma) \le \dot{r}_{A, T}(\sigma)$.
\end{lem}

\begin{proof}
This follows by induction on $\dot{r}_{A, T}(\sigma)$.
\end{proof}

\begin{lem}\label{lem:smaller_tree_bigger_laver_rank}
Let $T_0 \subseteq T_1$ be Laver trees with a common stem $\tau$, and let $A \subseteq T_1$.  Then for all $\sigma \supseteq \tau$ with $\sigma \in T_0$, $\dot{r}_{A, T_1}(\sigma) \le \dot{r}_{A, T_0}(\sigma)$.
\end{lem}

\begin{proof}
This follows by induction on $\dot{r}_{A, T_0}(\sigma)$.
\end{proof}

We would like to construct our trace $(V_n)_{n \in \omega}$ using $\dot{r}_{A,T}$ for various $A$ and $T$, but we lack an effective way to directly reason about $\dot{r}_{A,T}$.  We will show, however, that it suffices to reason about $(A, \+L, T)$ Laver rankings, where $\+L$ is a Harrison order. 

For $\+L$ an ill-founded linear order, it may be that no $(A,\+L, T)$ Laver ranking exists (in particular, the $\liminf$ required by the definition may not exist).  We can put various conditions on what such a ranking will look like when it does exist, though.

\begin{notation}
For a linear order $\+L$ with least and greatest elements $0$ and $\infty$, respectively, let $W(\+L)$ be the maximal well-founded initial segment of $\+L\setminus \{\infty\}$.  We identify $W(\+L)$ with an initial segment of the ordinals in the standard fashion (such identification being unique), writing, for example, $r(\sigma) = \alpha$ when $r(\sigma) \in W(\+L)$ and $[0, r(\sigma))_{\+L}$ has order type $\alpha$.
\end{notation}

\begin{lem}\label{lem:no_r_smaller}
For $r$ an $(A, \+L,T)$ Laver ranking, there is no $\sigma$ such that $r(\sigma) \in W(\+L)$ with $r(\sigma) < \dot{r}_{A,T}(\sigma)$.
\end{lem}

\begin{proof}
Suppose not.  We construct a sequence of strings $\sigma = \sigma_0 \subset \sigma_1 \subset \dots$ with $r(\sigma_{i+1}) < r(\sigma_i)$ for a contradiction, maintaining the inductive assumption that $r(\sigma_i) < \dot{r}_{A,T}(\sigma_i)$.  Clearly this holds for $i = 0$.

Suppose we have defined $\sigma_i$.  As $\dot{r}_{A,T}(\sigma_i) > r(\sigma_i)$, it is clear that $\dot{r}_{A,T}(\sigma_i) \neq 0$, and thus $\sigma_i \not \in A$.  As $r(\sigma_i) < \infty$, there are infinitely many $n$ such that $r(\sigma_i\cat n) < r(\sigma_i)$, whereas for almost every $n$, $\dot{r}_{A,T}(\sigma_i\cat n) \ge r(\sigma_i)$.  We can therefore define $\sigma_{i+1} = \sigma_i\cat n$ for an $n$ satisfying both of these.
\end{proof}

Note that this implies that if $\dot{r}_{A,T}(\sigma) = \infty$, then $r(\sigma) \not \in W(\+L)$.

\begin{lem}\label{lem:wf_ranks_same}
For $r$ an $(A,\+L,T)$ Laver ranking, if $\dot{r}_{A,T}(\sigma) < \infty$, then $r(\sigma) \in W(\+L)$ and $r(\sigma) = \dot{r}_{A,T}(\sigma)$.
\end{lem}

\begin{proof}
We show by induction on $\dot{r}_{A,T}(\sigma)$ that $r(\sigma) \in W(\+L)$ with  $r(\sigma) \le \dot{r}_{A,T}(\sigma)$.  For $\dot{r}_{A,T}(\sigma) = 0$, this is immediate.

For $0 < \dot{r}_{A,T}(\sigma) < \infty$, there are infinitely many $n$ with $\dot{r}_{A,T}(\sigma\cat n) < \dot{r}_{A,T}(\sigma)$.  By the inductive hypothesis, $r(\sigma\cat n) \in W(\+L)$ with $r(\sigma\cat n) \le \dot{r}_{A,T}(\sigma\cat n) < \dot{r}_{A,T}(\sigma)$.  Thus these $\sigma\cat n$ witness that $r(\sigma) \in W(\+L)$ with $r(\sigma) \le \dot{r}_{A,T}(\sigma)$.

The result then follows by Lemma~\ref{lem:no_r_smaller}.
\end{proof}

\begin{lem}
If $A$ and $T$ are hyperarithmetic, then for every $\sigma$ with $\dot{r}_{A,T}(\sigma) < \infty$, $\dot{r}_{A,T}(\sigma) < \omega_1^{ck}$.
\end{lem}

\begin{proof}
Suppose not.  We know that for some countable $\lambda$, $\dot{r}_{A,T}$ is an $(A, \lambda \cup \{\infty\}, T)$ Laver ranking with $\dot{r}_{A,T}(\sigma) \ge \omega_1^{ck}$.  By Lemma~\ref{lem:wf_ranks_same}, for any $(A, \+L,T)$ Laver ranking $r$, we have that $r(\sigma) \in W(\+L)$.  Thus a computable linear order $\+K$ is well founded if and only if there exists a countable linear order $\+L$, an $(A,\+L,T)$ Laver ranking $r$, and an embedding $\+K \hookrightarrow [0, r(\sigma))_{\+L}$.  However, this is a $\Sigma^1_1$ definition of well-foundedness.
\end{proof}

\begin{lem}\label{lem:harrison_rankings_exist}
If $\+H$ has order type $\omega_1^{ck}(1+\bQ) + 1$ with least and greatest elements $0$ and $\infty$, respectively, then for any hyperarithmetic $A$ and $T$, there is an $(A, \+H, T)$ Laver ranking.
\end{lem}

\begin{proof}
We can identify $W(\+H)$ with $\omega_1^{ck}$ and see that $\dot{r}_{A,T}$ is an $(A, \+H,T)$ Laver ranking.
\end{proof}

Letting a computable $\+H$ be as in Lemma~\ref{lem:harrison_rankings_exist} and continuing the identification of $W(\+H)$ with $\omega_1^{ck}$, we observe that for hyperarithmetic $A$ and $T$, the relation $\{ (\sigma, \alpha) : \dot{r}_{A,T}(\sigma) = \alpha\}$ has a $\Pi^1_1$ definition:
\[
\dot{r}_{A,T}(\sigma) = \alpha \iff \forall r \in LR(A,\+H,T)\, [r(\sigma) = \alpha].
\]
Moreover, for each $\beta < \omega_1^{ck}$, $\{ (\sigma, \alpha) : \alpha < \beta \wedge \dot{r}_{A,T}(\sigma) = \alpha\}$ is $(A\oplus T)^{(2\beta+3)}$-computable, and in fact uniformly so.  We can observe that if $\sigma \in T$ is such that $\dot{r}_{A, T}(\sigma) = \alpha < \infty$, then $(T^\sigma)_A$ is hyperarithmetic relative to $A$ and $T$ and uniformly so given $\alpha$.

We are now ready to prove Lemma~\ref{lem:tracing}.

\begin{proof}[Proof of Lemma~\ref{lem:tracing}]
Without loss of generality, we define $h(n) = n+1$.  (We are working in Baire space and thus can compress via pairing without affecting the compactness of the closure.)
Fix $(\tau_n)_{n \in \omega}$ an effective listing of $\omega^{<\omega}$ such that $\tau_i \subset \tau_j$ implies $i < j$, and thus $\tau_0$ is the empty string $\seq{}$.  Assume this listing has the property that for all $\tau$ and $b$, $\tau\cat b$ occurs earlier in the listing than $\tau\cat (b+1)$.

We perform a construction of length $\omega_1^{ck}$, enumerating partial functions $k: \omega^{<\omega} \to \omega^{<\omega}$ and $\ell_m: (m+1) \to \omega^m$ for $m \in \omega$.  We will also define hyperarithmetic trees $T_{n, m}$ for some $n \le m$. 
We adopt the notation $A_\rho = \{\sigma : \rho \subseteq \Phi^\sigma\}$, and for $n > 0$, $\tau^-_n$ is the parent of $\tau_n$.

We will ensure the following properties hold:
\begin{itemize}
\item $k$ is injective, order preserving and length preserving, and its domain is downwards closed.  Furthermore, for any $\tau$ such that $k(\tau)$ is defined, $k(\tau)$ majorizes $\tau$.
\item If $\tau_n$ majorizes $f$, then $k(\tau_n)$ is defined, as are $\ell_m(n)$ for all $m \ge n$.
\item For $n \le m$, $\ell_m(n)\converge$ if and only if $T_{n,m}$ is defined.
\item For $n \le m$, if $\ell_m(n)\converge$, then $\dot{r}_{A,T}(k(\tau_n)) < \infty$ for $T = T_{n,m}$ and $A = A_{\ell_m(n)}$.
\item For every $n \le m$ such that $T_{n,m}$ is defined, $T_{n, m}$ is a Laver tree with stem $k(\tau_n)$.
\item If $T_{n, m}$ is defined for $n < m$, then $T_{n,m-1}$ is defined and $T_{n, m} = (T_{n,m-1})_A$ for $A = A_{\ell_m(n)}$.
\item For any $n > 0$, suppose that $\tau_j = \tau_n^-$.  If $k(\tau_n)\converge$, then $T_{j, n-1}$ is defined and $k(\tau_n) \in T_{j, n-1}$.  If $T_{n,n}$ is defined, then $T_{n,n} = \left(T_{j, n-1}^{k(\tau_n)}\right)_A$ for $A = A_{\ell_n(n)}$.
\end{itemize}
This essentially describes the construction.  We begin with $k(\seq{}) = \seq{}$, $\ell_0(0) = \seq{}$, and $T_{0, 0} = \omega^{<\omega}$.

For $n > 0$, let $\tau_j = \tau_n^-$.  We wait to see if $k(\tau_j)$ and $T_{j, n-1}$ are ever defined.  If this happens, we then define $k(\tau_n)$ to be the $n$th child of $k(\tau_j)$ on $T_{j, n-1}$.

For $n \le m$ with $m > 0$, if $m = n$, we then let $\tau_j = \tau_n^-$ and let $T = T_{j, n-1}^{k(\tau_n)}$; if $m > n$, we let $T = T_{n, m-1}$  (including when $n = 0$).  We wait until $T$ is defined and then search for $\rho \in \omega^m$ such that $\dot{r}_{A, T}(k(\tau)) < \infty$ for $A = A_\rho$; if we see such a $\rho$, we define $\ell_m(n) = \rho$ and $T_{n, m} = T_A$ for $A = A_\rho$.

Most of the stated properties of the construction are immediate.

\begin{claim}
For $m \le n$, if $\ell_m(n)\converge$, then $\dot{r}_{A,T}(k(\tau_n)) < \infty$ for $T = T_{n,m}$ and $A = A_{\ell_m(n)}$.
\end{claim}

\begin{proof}
This is immediate from the construction except in the case $m = n = 0$.  We handle this case by observing that $A_{\ell_0(0)} = A_{\seq{}} = \omega^{<\omega}$, and so $\dot{r}_{A, T}(k(\tau_0)) = 0$.
\end{proof}

\begin{claim}
$k$ is injective, and if $k(\tau)$ is defined, then $k(\tau)$ majorizes $\tau$.
\end{claim}

\begin{proof}
This follows by induction on $|\tau|$.  The case $\tau = \seq{}$ is immediate.

Suppose $k(\tau_j)$ is defined, and let $\tau_n = \tau_j\cat b$.  Note that $n > b$.  Then, if $k(\tau_n)$ is defined, it is the $n$th child of $k(\tau_j)$ on $T_{j, n-1}$, so $k(\tau_n) = k(\tau_j)\cat c$ for some $c \ge n > b$.  By induction, $k(\tau_n)$ majorizes $\tau_n$.

Now suppose $\tau_{n'} = \tau_j\cat b'$ is another child of $\tau_j$ such that $n' > n$ and thus $b' > b$.  If $k(\tau_{n'})$ is defined, it is the $n'$th child of $k(\tau_j)$ on $T_{j, n'-1} \subseteq T_{j, n-1}$, and thus it is the $(\ge n')$th child of $\tau_j$ on $T_{j, n-1}$.  It follows that $k(\tau_{n'})\neq k(\tau_n)$.  By induction, $k$ is injective.
\end{proof}

\begin{claim}
If $\tau_n$ majorizes $f$, then $k(\tau_n)$ is defined, as is $\ell_m(n)$ for all $m \ge n$.
\end{claim}

\begin{proof}
This follows by induction on $|\tau_n|$. Note that $k(\tau_0)$ and $\ell_0(0)$ are explicitly defined at the start of the construction.

For $n > 0$ with $\tau_n$ majorizing $f$, let $\tau_j = \tau_n^-$.  By induction, $k(\tau_j)$ and $T_{j, n-1}$ are eventually defined, so $k(\tau_n)$ is defined.

Now, we do a subinduction on $m \ge n$ (if $n = 0$, we restrict ourselves to $m > n$).  Let $T$ be as given in the definition of $\ell_m(n)$.  By induction, $T$ is eventually defined.  Let $\pi = \ell_{m-1}(n)$ if $m > n$ and $\pi = \ell_{j}(n-1)$ if $m = n$.  Note that $[T]$ is covered by $A_\pi$ by construction.

As $\+P$ has compact closure, there are only finitely many $y$ with $[\pi\cat y] \cap \+P \neq \emptyset$.  We list these as $y_0, y_1, \dots, y_{s-1}$.  Let $A_i = A_{\pi\cat y_i}$ and, towards a contradiction, suppose $\dot{r}_{A_i,T}(k(\tau_n)) = \infty$ for every $i < s$.  Now we define a sequence of Laver trees $T = Q_0 \supseteq Q_1 \supseteq \dots \supseteq Q_s$ by $Q_{i+1} = (Q_i)_{A_i}$.    By Lemmas~\ref{lem:smaller_set_bigger_laver_rank} and~\ref{lem:smaller_tree_bigger_laver_rank}, $\dot{r}_{A_i, Q_i}(k(\tau_n)) = \infty$ for all $i < s$, and thus $[Q_s]$ is disjoint from $A_i$ for all $i < s$.  However, $\tau_n$ majorizes $f$ and $k(\tau_n)$ majorizes $\tau_n$, so $k(\tau_n)$ can be extended to a $g \in [Q_s]$ majorizing $f$.  Then $\Phi^g \in \+P$ by assumption and $g \in [T]$, so $g$ meets $A_\pi$ and thus $\pi \subset \Phi^g$.  But $g$ avoids all $A_i$, and thus $\pi\cat y_i \not \subset \Phi^g$ for any $i < s$, giving us a contradiction.

Thus there is some $y_i$ such that $\dot{r}_{A_i, T}(k(\tau_n)) < \infty$, and so $\rho = \pi\cat y_i$ is as desired for the definition of $\ell_m(n)$.
\end{proof}

Thus our construction has the stated properties.  Now we define
\[
V_m = \{ \ell_m(n) : n \le m\}.
\]
These are uniformly $\Pi^1_1$ and of the correct size.  We can also define
\[
T = \{ k(\tau) : \text{$\tau$ majorizes $f$}\}
\]
and observe that $T$ is a stemless Laver tree.

\begin{claim}
For every $g \in [T]$, $(V_m)_{m \in \omega}$ is an $h$-trace of $\Phi^g$.
\end{claim}
\begin{proof}
Fix $n \le m$ largest with $\tau_n \subset g$, and let $\rho = \ell_m(n)$.  Then, by construction, $g \in [T_{n, m}]$ and thus $g$ meets $A_\rho$. This guarantees that $\Phi^g\uhr{m} = \rho \in V_m$.
\end{proof}

This completes the proof of Lemma~\ref{lem:tracing}.
\end{proof}

\subsection{A class without $\Pi^1_1$ traces} At this point, we can finally complete the separation of uniform highness for descending sequences and highness for isomorphism. Suppose $\+H$ is a Harrison order and $(X_a)_{a \in \+H}$ is a jump structure on $\+H$.  We wish to show that there is no $\Pi^1_1$ $h$-trace for $(X_a)_{a \in \+H}$ for an appropriately chosen computable order $h$, not depending on $(X_a)_{a \in \+H}$.

\begin{lem}\label{lem:approximating_finite_pi11}
Suppose $V$ is a finite set with $|V| \le n$.  Then there are sets $(U_a)_{a \in H}$ satisfying:
\begin{itemize}
\item For all $a$, $|U_a| \le n$;
\item For all $a < b$, $U_a \subseteq U_b$;
\item $V = \bigcup_{a \in W(\+H)} U_a$; and
\item Each $U_a$ is enumerable from $X_a$ uniformly in $a$, $n$, and a $\Pi^1_1$ index for $V$.
\end{itemize}
\end{lem}

\begin{proof}
We fix uniformly computable trees $T_i \subseteq \omega^{<\omega}$ such that $T_i$ is well founded if and only if $i \in V$.  Define $R(i, a)$ to be the relation ``$X_a$ computes a tree-rank function from $T_i$ to $\+H_{< a}$'' (a standard tree-rank function, not a Laver ranking).  Note that $\{ i : R(i, a)\}$ is uniformly $\Sigma^0_3(X_a)$ and that $i \in V$ if and only if there is some $a \in W(\+H)$ satisfying $R(i,a)$.

Now we define
\[
W_a = \left\{ i : R(i, a) \wedge \neg \exists j_0 < \dots < j_{n-1} \bigwedge_{s < n} [(j_s < i \wedge R(j_s, a)) \vee (\exists b < a) R(j_s, b)]\right\}.
\]
Define $\widehat{W}_a = \bigcup_{b \le a} W_b$.  We claim that $(\widehat{W}_a)_{a \in \+H}$ is as desired except for not being uniformly enumerable from $X_a$.

First, towards a contradiction, we fix some $a$ with $|\widehat{W}_a| > n$ and then fix distinct $i_0, \dots, i_n \in \widehat{W}_a$.  For each $s \le n$, we fix $b_s \le a$ with $i_s \in W_{b_s}$ and let $b = \max\{b_s : s \le n\}$.  Now we let $i = \max (W_b \cap \{i_0, \dots, i_n\})$ and derive a contradiction from $i \in W_b$.

Clearly, for $a \in W(\+H)$, we have that $W_a \subseteq V$.  For $i \in V$, let $a$ be least such that $R(i,a)$ holds.  Then $\bigcup_{b < a} W_b \subseteq V \setminus \{i\}$, and so $|\bigcup_{b < a} W_b| < n$.  This gives us that $i \in W_a$, and it follows that $V = \bigcup_{a \in W(H)} \widehat{W}_a$.

Now we observe that $\widehat{W}_a$ is $\Sigma^0_4(X_a)$ uniformly in $a$, $n$, and a $\Pi^1_1$ index for $V$, so we define
\[
U_a = \bigcup_{b+3 \le a} \widehat{W}_b.\qedhere
\]
\end{proof}

\begin{prop}\label{thm:jump_structures_untraced}
Let $\+H$ be a Harrison order.  There is a computable order $h$ such that no jump structure on $\+H$ has a $\Pi^1_1$ $h$-trace.
\end{prop}

\begin{proof}
We fix some $b+1 \in \+H\setminus W(\+H)$ nonuniformly and let $J^Z(e) = \Phi^Z_e(e)$ for all $e \in \omega$ and $Z \in 2^{\omega}$.  By the recursion theorem, there is an injective computable sequence $e_{i, j}$ for $i \in \omega$ and $j\le i$ such that we control the behavior of each $J(e_{i,j})$.  We let $r_i = \max\{ \seq{b+1, e_{i, j}}+1 : j \le i\}$ and define a computable order $h$ such that $h(r_i) \le i+1$ for all $i$.

Now we let $((V_m^i)_{m \in \omega})_{i \in \omega}$ be a uniform listing of all $\Pi^1_1$ $h$-traces.  We let $(U_a^i)_{a \in H}$ be as in Lemma~\ref{lem:approximating_finite_pi11} for $V = V_{r_i}^i$ and $n = h(r_i)$ and define $J^Z(e_{i,j})$ as follows: we first attempt to enumerate $U_b^i$, working under the assumption that $Z = X_b$; and, letting $\sigma_0, \sigma_1, \dots$ be the elements we enumerate in order of enumeration, we define $J^Z(e_{i,j})\converge$ if and only if $\sigma_j$ exists and $\sigma_j(e_{i,j}) = 0$.

Now we suppose $(X_a)_{a \in \+H}$ is a jump structure on $\+H$.  Towards a contradiction, we further suppose $(V_n^i)_{n \in \omega}$ traces $X = \{ \seq{a, x} : x \in X_a\}$, and we fix $\sigma \in V_{r_i}^i$ an initial segment of $X$.  Then $V_{r_i}^i \subseteq U_b^i$ and $U_b^i$ is enumerated by $X_b$, so $\sigma = \sigma_j$ for some $j < h(r_i) \le i+1$.  Then
\[
\seq{b+1, e_{i, j}} \in X \iff e_{i,j} \in X_{b+1} \iff J^{X_b}(e_{i,j})\converge \iff \sigma(e_{i,j}) = 0,
\]
which contradicts $\sigma$ being an initial segment of $X$.
\end{proof}

\begin{lem}\label{lem:dominating_to_majorizing_for_Pi01}
Suppose $\+P$ is a $\Pi^0_1$ class and $f \in \omega^\omega$ is such that every function dominating $f$ computes an element of $\+P$.  Then there is a $\Phi$ and an $f'$ such that for every function $g$ majorizing $f'$, $\Phi^g$ is an element of $\+P$.
\end{lem}

\begin{proof}
We perform Hechler forcing: a condition is a pair $(\sigma, g)$ such that $\sigma \in \omega^{<\omega}$, $g \in \omega^\omega$, and $\sigma \subset g$; we say that $(\sigma_0, g_0) \le (\sigma_1, g_1)$ if $\sigma_0 \supseteq \sigma_1$ and $g_0$ majorizes $g_1$.

If $F$ is a sufficiently generic filter, $g_F = \bigcup_{(\sigma, g) \in F} \sigma$ is a total function majorizing every $g$ with $(\sigma, g) \in F$.  Note that for any condition $(\sigma, g)$, if we define
\[
g'(x) = \left\{\begin{array}{cl} \sigma(x) & x < |\sigma|\\ \max\{f(x), g(x)\} & \text{otherwise,}\end{array}\right.
\]
then $g'$ dominates $f$ and $(\sigma, g') \le (\sigma, g)$.  So for a sufficiently generic filter $F$, $g_F$ will dominate $f$.  We fix $\Psi$ such that $\Psi^{g_F} \in \+P$.  Now we consider the following set of conditions:
\[
D = \{ (\sigma, g) : [\Psi^\sigma] \cap \+P = \emptyset \} \cup \{ (\sigma, g) : \exists n\, \forall (\sigma', g') \le (\sigma, g)\, [\Psi^{\sigma'}(n)\diverge]\}.
\]
$F$ must not meet $D$, so there is some $(\tau, f') \in F$ with no extension in $D$.  We claim this is our desired $f'$.

We define a functional $\Phi$ as follows: given $g$ majorizing $f'$, we define a sequence $\tau = \sigma_0 \subset \sigma_1 \subset \dots$ such that $\Psi^{\sigma_n}\uhr{n}\converge$ and for all $x$ with $|\sigma_n| \le x < |\sigma_{n+1}||$, $\sigma_{n+1}(x) \ge g(x)$.  Then we define $\Phi^g = \bigcup_n \Psi^{\sigma_n}$.

Now, towards a contradiction, we suppose $n$ is least such that there is no $\sigma_{n+1}$.  Then for
\[
g'(x) = \left\{\begin{array}{cl} \sigma_n(x) & x < |\sigma_n|\\ g(x) & \text{otherwise,}\end{array}\right.
\]
$(\sigma_n, g') \in D$ is an extension of $(\tau, f')$, giving us a contradiction.  Thus $\Phi^g$ is total.

Again towards a contradiction, we suppose $\Phi^g \not \in \+P$.  Then there is some least $n$ such that $[\Psi^{\sigma_n}] \cap \+P = \emptyset$.  For the same $g'$ as previously, $(\sigma_n, g') \in D$ is an extension of $(\tau, f')$, once more giving us a contradiction.
\end{proof}

Let $\+H$ be a Harrison order which admits a jump structure.  Then the collection of jump structures on $\+H$ is a $\Pi^0_2$ subset of $2^\omega$, and there is a $\Pi^0_1$ $\+P \subseteq \omega^\omega \times 2^\omega$ which projects onto $\+H$.

\begin{thm}\label{thm:jump_structures_not_fast_growing_property}
Let $\+H$ be a Harrison order which admits a jump structure, and let $\+P$ be a $\Pi^0_1$ class which projects onto the jump structures of $\+H$. Then for any $f \in \omega^\omega$, there is a $g$ dominating $f$ which does not compute an element of $\+P$.
\end{thm}

\begin{proof}
Suppose not.  Then, by Lemma~\ref{lem:dominating_to_majorizing_for_Pi01}, there are $\Phi$ and $f$ such that for every $g$ majorizing $f$, $\Phi^g \in \+P$.  By composing with projection, we get $\Psi$ such that for every $g$ majorizing $f$, $\Psi^g$ is a jump structure on $\+H$.  As the collection of jump structures is a subset of $2^\omega$, it has compact closure, and by Lemma~\ref{lem:tracing}, for every computable order $h$ there is a jump structure on $\+H$ with a $\Pi^1_1$ $h$-trace.  This contradicts Proposition~\ref{thm:jump_structures_untraced}.
\end{proof}

\begin{lem}\label{lem:descending_sequences_fast_growing_property}
If $\+L$ is a computable linear order (where we identify the domain with $\omega$) and $f$ is a descending sequence (thought of as an element of $\omega^\omega$), then any $g$ majorizing $f$ uniformly computes a descending sequence through $\+L$.
\end{lem}

\begin{proof}
Let $a_0$ be the $\+L$-rightmost of the first $g(0)$ elements of $\+L$, and, given $a_n$, define $a_{n+1}$ to be the $\+L$-rightmost of the first $g(n+1)$ elements of $\+L$ which are to the left of $a_n$.  Inductively, $f(n) \le_{\+L} a_n$, and so a candidate for $a_{n+1}$ exists; in particular, $f(n+1)$ is such an element.
\end{proof}

\begin{cor}\label{cor:separating_uniformly_high_desc_from_high_iso}
There is a degree which is uniformly high for descending sequences but not high for isomorphism.
\end{cor}

\begin{proof}
Let $f$ be a total function such that for every $e$ such that $\+L_e$ is ill founded, $f^{[e]}$ is a descending sequence.  By Theorem~\ref{thm:jump_structures_not_fast_growing_property}, there is a $g$ majorizing $f$ which is not high for paths, and by Lemma~\ref{lem:descending_sequences_fast_growing_property}, this $g$ is uniformly high for descending sequences.
\end{proof}

We will strengthen this result in Theorem~\ref{thm:separating_uniformly_high_desc_improved}. 

Though highness, even uniform highness, for descending sequences does not suffice to imply highness for isomorphism, the addition of a different property provides the missing strength, although this property itself seems rather weak (see Proposition~\ref{scLFO1CK}).

\begin{thm}\label{prop:high_for_descending_and_Scott_complete_implies_high_for_isomorphism} A degree $\mathbf{d}$ is high for isomorphism if and only if it is both high for descending sequences and Scott complete.
\end{thm}

\begin{proof}
Given isomorphic computable structures $\sA$ and $\sB$, we form a structure $\sA \odot \sB$ by expanding the cardinal sum of $\sA$ and $\sB$ with two new constant symbols $a,b$, adding new elements to instantiate them, and adding a new relation $S$ such that $aSx$ for exactly those $x \in \sA$ and $bSx$ for exactly those $x \in \sB$.  We take a Harrison order $\sH$ whose well-founded part $W(\+H)$ is Turing equivalent to $\sO$, a standard operation described in \cite{Fr76, GHKS04}, and let $(\leq_x : x \in \sH)$ be the Scott analysis computed by $\mathbf{d}$.

We now consider the points $a,b$.  As $\+A \cong \+B$, $a \cong_\alpha b$ for all ordinals $\alpha$, and thus $a \cong_x b$ for every $x \in W(\+H)$.  If for all $x \in \+H\setminus W(\+H)$ we have $a \ncong_x b$, then $\*d$ easily computes $W(\+H)$ and thus computes $\sO$, making $\*d$ high for isomorphism.

Thus we can assume that there is some $x \in \sH\setminus W(\+H)$ such that $a \cong_x b$ and take $\sH' = \sH\upharpoonright_x$.  Since $\mathbf{d}$ is high for descending sequences, we find a $\mathbf{d}$-computable descending sequence $(s_i : i \in \mathbb{N}) \subseteq \sH'$.  We will use $(\leq_{s_i} : i \in \mathbb{N})$ to compute an isomorphism $f:\sA \to \sB$.  We construct a sequence $(\tilde{f}_t : t \in \mathbb{N})$ of finite partial isomorphisms; for notational convenience, we will also construct $\delta_t$ a tuple of the elements of $\dom(f_t)$ and $\gamma_t$ a tuple of the elements of $\text{range}(f_t)$.

We let $\tilde{f}_{0} = \{(a,b)\}$, denoting $\delta_{0} = (a)$ and $\gamma_{0} = (b)$.  Now, at stage $2t$, we take the first $c \in \sA \setminus \dom(\tilde{f}_{2t})$ and find (using oracle $\mathbf{d}$) some $d$ such that $\delta_{2t} \cat c \leq_{s} \gamma_{2t} \cat d$ for $s = s_{2t}$.  We set $\tilde{f}_{2t+1} = \tilde{f}_{2t} \cup \{(c,d)\}$ and make $\delta_{2t+1} = \delta_{2t} \cat c$ and $\gamma_{2t+1} = \gamma_{2t} \cat d$.  We act symmetrically at stage $2t+1$ so that both $\sA$ and $\sB$ are exhausted.  We now define $\tilde{f} = \bigcup\limits_{t \in \mathbb{N}} \tilde{f}_t$ and $f = \tilde{f} \setminus \{(a,b)\}$.  Finally, we can see that $f:\sA \to \sB$ is an isomorphism.

The converse follows from the discussion at the conclusion of Section \ref{sec:scott} and at the beginning of the present section.
\end{proof}

\begin{cor}
There is a degree which is uniformly high for descending sequences but not Scott complete, and conversely.\end{cor}

\begin{proof} By Corollary~\ref{cor:separating_uniformly_high_desc_from_high_iso}, uniform highness for descending sequences does not imply highness for isomorphism, and thus cannot imply Scott completeness.

As observed in Section~\ref{sec:scott}, Scott completeness does not imply highness for isomorphism, and thus cannot imply highness for descending sequences.  Alternatively, this follows from the same reasoning employed in Proposition~\ref{scNOTjc}.
\end{proof}

\section{Reticence}\label{sec:reticence}

It is standard in computability to distinguish between positive information (e.g., an enumeration) and complete information.  A structural analogue of this distinction is what we call ``reticence."  By this we mean that a degree that is, for instance, uniformly high for isomorphism "in the reticent sense" should not only compute isomorphisms where they exist but refrain from computing things that seem to be isomorphisms when no isomorphism exists. We present a formal definition here.

\begin{dfn}
Let $\*d$ be a Turing degree which is uniformly high for a given notion $\+N$.  That is, there is some $I \subseteq \omega$ (depending on $\+N$) and a $\*d$-computable function $F$ such that if $i \in I$, then $F(i) \in \omega^\omega$ has some $\+N$-appropriate relation to $i$, (e.g., uniformly high for paths, where $I$ is the set of indices for nonempty $\Pi^0_1$ classes and $F(i) \in \+P_i$).

We say that $\*d$ is {\em uniformly high for $\+N$ in the reticent sense} if $F$ can be chosen such that for every $i \not \in I$, the partial function $F(i)$ diverges for all inputs.
\end{dfn}

Of course, being uniformly high for some notion in the reticent sense implies being uniformly high in the normal sense.  Possibly the most interesting feature of reticence from our perspective is that our various highness properties, apparently distinct without reticence, collapse when this further requirement is introduced.

\begin{thm}\label{thm:reticence_all_the_same}
The following properties of a degree $\mathbf{d}$ are equivalent:
\begin{enumerate}
    \item $\mathbf{d}$ enumerates all $\Sigma^1_1$ sets.
    \item $\mathbf{d}$ is uniformly high for paths in the reticent sense.
    \item $\mathbf{d}$ is uniformly high for isomorphism in the reticent sense.
    \item $\mathbf{d}$ is uniformly high for descending sequences in the reticent sense.
    \item $\mathbf{d}$ is uniformly high for tight descending sequences.
    \item $\mathbf{d}$ is uniformly high for tight descending sequences in the reticent sense.
\end{enumerate}
\end{thm}

\begin{proof}
The proof is organized as follows:
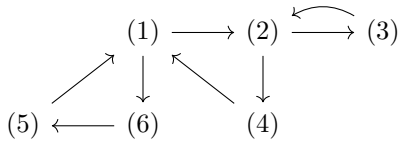
\begin{center}
\begin{tikzcd}
& (1) \arrow[r] \arrow[d] & (2) \arrow[r]\arrow[d] & (3) \arrow[l,bend right] \\
(5)\arrow[ur] & (6) \arrow[l] & (4) \arrow[ul] & 
\end{tikzcd}
\end{center}

(1 $\Rightarrow$ 2)  Let $\+P_e$ be an effective listing of $\Pi^0_1$ classes in Baire space.  By assumption, $\*d$ enumerates $X = \{ (\sigma, e) : [\sigma]\cap \+P_e \neq \emptyset\}$.

Given $e$, we search for a string $\sigma_0$ with $(\sigma_0, e) \in X$.  Having found such a $\sigma_0$, we define $F(e)\uhr{|\sigma_0|} = \sigma_0$.  We then search for a string $\sigma_1 \in X$ strictly extending $\sigma_0$, and continue in this fashion.

If $\+P_e$ is empty, then there is no such $\sigma_0$, and so we will leave $F(e)$ everywhere divergent.  On the other hand, if $\+P_e$ is nonempty, then we will find our sequence of $\sigma_i$s and use them to define $F(e) \in \bigcap_i [\sigma_i] \subseteq \+P_e$.

Thus $f$ witnesses that $\*d$ is uniformly high for paths in the reticent sense.

\smallskip

($2\Rightarrow 3)$ Given computable structures $\+M_i$ and $\+M_j$, $\{ (f, f^{-1}) : \+M_i \cong_f \+M_j\}$ is a $\Pi^0_1$ class.

\smallskip

($3\Rightarrow 2$) In the proof of Proposition~\ref{prop:high_for_iso_vs_paths}, we gave a uniform procedure for transforming a $\Pi^0_1$ class into a pair of computable structures such that the structures are isomorphic if and only if the class is nonempty, and every isomorphism computes an element of the class.  The result follows.

\smallskip

($2\Rightarrow 4$) The infinite descending sequences through a computable linear order form a $\Pi^0_1$ class, and an index for the $\Pi^0_1$ class can be effectively obtained from an index for the linear order.

\smallskip

($1\Rightarrow 6$) Let $\+L_e$ be an effective listing of computable linear orders.  By assumption, $\*d$ enumerates $X = \{(a, e) : \text{$a$ is in the ill-founded part of $\+L_e$}\}$.  We fix such an enumeration $(X_s)_{s \in \omega}$ and assume $X_s$ is finite for every $s$.  Given $e$, we search for $a_0 \in \+L_e$ with $(a_0, e) \in X$; if we find such an $a_0$ at stage $s$, we make sure to choose it such that $a_0$ is the leftmost element of $\+L_e$ with $(a_0, e) \in X_s$.  We then search for an $a_1 \in \+L_e$ with $(a_1, e) \in X$ and $a_1 <_{\+L_e} a_0$; again, we choose $a_1$ to be the leftmost eligible element at the stage we find it.  We continue in this fashion and define $F(e) = \seq{a_0, a_1, \dots}$

Clearly if $\+L_e$ is well founded, then we never find $a_0$, and so $F(e)$ diverges for all inputs.  Suppose now that $b$ is an element of the ill-founded part of $\+L_e$.  There there is some $s$ with $(b, e) \in X_s$, and so any $a_n$ chosen on or after stage $s$ will be $\le_{\+L_e} b$ by construction.
\smallskip

($6\Rightarrow 5$)  This is immediate.

\smallskip

($4 \Rightarrow 1$)  Given a $\Sigma^1_1$ set $Y$,  there is a computable sequence of linear orders $(\+L_{e_n})_{n \in \omega}$ such that $\+L_{e_n}$ is ill founded if and only if $n \in Y$.  Therefore, $n \in Y$ if and only if $F(e_n)$ produces any output, which gives a $\*d$-computable enumeration.

\smallskip

($5\Rightarrow 1$)  Fix a $\Sigma^1_1$ set $Y$.  Then there is a computable sequence of linear orders $(\+L_{e_n})_{n \in \omega}$ such that $\+L_{e_n}$ is ill founded if and only if $n \in Y$.  Now fix a Harrison order $H$ and, for each $n$, define $i_n$ so that $\+L_{i_n} = \+L_{e_n} + H$.  Then $\+L_{i_n}$ will be ill founded and $F(i_n)$ will produce a tight descending sequence through it. Then $n \in Y$ if and only if $F(i_n)$ enumerates an element of the $\+L_{e_n}$ part of $\+L_{i_n}$, which gives a $\*d$-computable enumeration.
\end{proof}

The striking thing about this result is the light it sheds on the known separations among the ``nonreticent" versions of these properties.  The separation for any two of these properties must somehow be driven by how the degrees in question handle the ``dark matter" of, e.g., the nonisomorphic structures or non-ill-founded orders.  To the present authors, at least, this feature was not obvious from our initial proofs of separation.

Theorem~\ref{thm:reticence_all_the_same} and Corollary~\ref{cor:separating_uniformly_high_desc_from_high_iso} immediately separate uniform highness for descending sequences from the reticent notion.  We can also separate uniform highness for isomorphism from the reticent notions.

\begin{thm}\label{thm:separating_uniformly_high_desc_improved}
There is a $\*d$ with $\*d' \equiv_T \+O$ that is uniformly high for descending sequences but not high for isomorphism, and such that there is no $\*b$ that is uniformly high for descending sequences in the reticent sense such that $\*d \le_T\*b$ and  $\+O \not \le_T \*b$.
\end{thm}

\begin{proof}
Fix $(\+P_e)_{e \in \omega}$ an effective listing of all $\Pi^0_1$ classes in Baire space.  Now define $c(0) = -1$ and, for $n>0$, define $c(n)$ to be the largest $x$ such that $\seq{x,y} < n$ for some $y$ where $\seq{\cdot, \cdot}$ is the standard pairing function.  As before, we write $J$ for the universal jump functional.  We define a forcing notion and then carefully build a sequence through the notion.

Our conditions will be triples $(\tau, f, n)\in \omega^{<\omega}\times \omega^\omega \times \omega$ satisfying:
\begin{itemize}
\item $\tau$ majorizes $f$;
\item $c(|\tau|) < n$;
\item $f$ is strongly hyperlow; and
\item For $e < n$, if $\+P_e$ is nonempty, then $f^{[e]}$ majorizes an element of $\+P_e$.
\end{itemize}
Condition $(\sigma, g, m)$ extends condition $(\tau, f, n)$ if:
\begin{itemize}
\item $\sigma \supseteq \tau$;
\item $g$ majorizes $f$;
\item $m \ge n$; and
\item For all $x < |\tau|$, $g(x) = f(x)$.
\end{itemize}
If $(\sigma, g, m)$ extends $(\tau, f, n)$ and, in addition, $g^{[e]} = f^{[e]}$ for all $e < r$, then we say that $(\sigma, g, m)$ {\em $r$-extends} $(\tau, f, n)$.  Thus $0$-extension is simply extension.

We build an $\+O$-computable sequence of conditions $(\tau_0, f_0, n_0), (\tau_1, f_1, n_1), \dots$, each extending the previous.  Our degree $\*d$ will be the degree of $G = \bigcup_s \tau_s$.  We begin by setting $\tau_0$ to be the empty string, $f_0$ to be the constant $0$ function, and $n_0 = 0$.

\smallskip

{\em At stage $3s$.} Suppose we have constructed $(\tau_{3s}, f_{3s}, n_{3s})$.  Let $e = n_{3s}$.  An oracle for $\+O$ can determine whether $\+P_{e}$ is empty.  If so, we define $f_{3s+1} = f_{3s}$.  If not, we fix some $h \in \+P_{e}$ such that $f_{3s} \oplus h$ is strongly hyperlow; we define $f_{3s+1}^{[e]}(x) = \max\{ f_{3s}^{[e]}(x), h(x)\}$ and $f_{3s+1}(x) = f_{3s}(x)$ elsewhere.  In either case, we let $\tau_{3s+1} = \tau_{3s}$ and $n_{3s+1} = e+1$. We have now forced that $G^{[e]}$ dominates an element of $\+P_e$ if the latter is nonempty.  Note that our new condition $e$-extends the previous condition.

\smallskip

{\em At stage $3s+1$.} Suppose we have constructed $(\tau_{3s+1}, f_{3s+1}, n_{3s+1})$.  Let $e = n_{3s+1}$ and consider the set of $g$ satisfying the following:
\begin{itemize}
\item $g$ majorizes $f_{3s+1}$;
\item For all $e < n_{3s+1}$, $g^{[e]} = f_{3s+1}^{[e]}$; and
\item For all $\sigma \supseteq \tau_{3s+1}$ majorizing $g$, $J^\sigma(s)\diverge$.
\end{itemize}

This is a $\Pi^0_1(f_{3s+1})$ class, and so $\+O$ can determine whether it is empty.  If it is nonempty, let $n_{3s+2} = n_{3s+1}$, $\tau_{3s+2} = \tau_{3s+1}$, and $f_{3s+2}$ be some strongly hyperlow element of the class; we have now forced $J^G(s)\diverge$.

If instead the class is empty, we begin building a sequence of extensions: 
\[
(\tau_{3s+1}, f_{3s+1}, n_{3s+1}) = (\rho_0, h_0, m_0), (\rho_1, h_1, m_1), (\rho_2, h_2, m_2), \dots
\]
Each term is constructed from the previous term following the process for stage $3s$, and so each $(\rho_{k+1}, h_{k+1}, m_{k+1})$ $m_k$-extends $(\rho_k, h_k, m_k)$.  After constructing the $k$th term, we ask $\+O$ if there is a $\sigma$ satisfying the following:
\begin{itemize}
\item $\sigma \supseteq \tau_{3s+1}$;
\item $c(|\sigma|) < m_k$;
\item $\sigma$ majorizes $h_k$; and
\item $J^\sigma(s)\converge$.
\end{itemize}
If there is no such $\sigma$ for any $k$, then $H = \lim_k h_k$ is an element of the earlier $\Pi^0_1(f_{3s+1})$ class which we assumed was empty, giving us a contradiction.  Thus eventually we locate such a $k$ and witnessing $\sigma$.  We then define $\tau_{3s+2} = \sigma$, $f_{3s+2} = h_k$ and $n_{3s+2} = m_k$.  We have now forced $J^G(s)\converge$.

\smallskip

{\em At stage $3s+2$.}  Suppose we have constructed $(\tau_{3s+1}, f_{3s+1}, n_{3s+1})$.  Fix $T \subseteq \omega^{<\omega}$ a computable tree such that $[T]$ is the $\Pi^0_1$ class guaranteed by Lemma~\ref{thm:jump_structures_not_fast_growing_property}.  We may assume that for all oracles $X$ and all $t$, if $(\Phi_s^X\uhr{t})\converge$, then $(\Phi_s^X\uhr{t}) \in T$.  We again build a sequence of extensions:
\[
(\tau_{3s+2}, f_{3s+2}, n_{3s+2}) = (\rho_0, h_0, m_0), (\rho_1, h_1, m_1), (\rho_2, h_2, m_2), \dots
\]
Each term is again constructed from the previous term following the process for stage $3s$, and so each $(\rho_{k+1}, h_{k+1}, m_{k+1})$ $m_k$-extends $(\rho_k, h_k, m_k)$.  After constructing the $k$th term, we ask $\+O$ if there is a $g \in \omega^\omega$, a $\sigma \in \omega^{<\omega}$ and a $t \in \omega$ satisfying the following:
\begin{itemize}
\item $g$ majorizes $h_k$;
\item $(\sigma, g, m_k)$ extends $(\tau_{3s+2}, f_{3s+2}, n_{3s+2})$; and
\item There is no $\sigma' \supseteq \sigma$ majorizing $g$ with $(\Phi_s^{\sigma'}\uhr{t})\converge$.
\end{itemize}
Towards a contradiction, we suppose there are never any such $g$, $\sigma$, and $t$ for any $k$.  Let $H = \lim_k h_k$.  Given an oracle $g$ majorizing $H$, we may assume $g$ extends $\tau_{3s+2}$.  Consider the following algorithm:
\begin{itemize}
\item Let $\sigma_0 = \tau_{3s+2}$;
\item Given $\sigma_i$, search for a $\sigma_{i+1} \supseteq \sigma_i$ that majorizes $g$ with $\Phi_s^{\sigma_{i+1}}(i)\converge$.
\end{itemize}
If this algorithm finds a $\sigma_i$ for every $i$, it computes a path through $T$ via $\Phi_s$, but by Lemma~\ref{thm:jump_structures_not_fast_growing_property}, there is some $g$ majorizing $H$ for which it fails.  Fix such a $g$ and let $i$ be least such that the algorithm fails to define $\sigma_{i+1}$.  Furthermore, let $k$ be such that $c(|\sigma_i|) < m_k$.  Then this $g$, $\sigma = \sigma_i$, and $t = i$ are as desired for $k$, contrary to assumption.

Therefore, eventually we locate a $k$ for which $g$ and $\sigma$ exist.  We fix such $g$ and $\sigma$ with $g$ strongly hyperlow.  We set $(\tau_{3s+3}, f_{3s+3}, n_{3s+3}) = (\sigma, g, m_k)$, and we have now forced that $\Phi_s^G \not \in [T]$.

\smallskip

By construction, $G^{[e]}$ majorizes an element of $\+P_e$ whenever the latter is nonempty.  It follows that $\*d$ is uniformly high for descending sequences.  It also follows that $\*d$ enumerates $\+O$: a $\Pi^0_1$ class $\+P_e$ is empty if and only if $\+P_e \cap \{ g : \forall x\, [g(x) \le G^{[e]}(x)]\}$ is empty, and the latter is effectively compact relative to $G$.  We force the jump, so $\*d' \le_T \+O$ and thus $\*d' \equiv_T \+O$.  We also ensure that $\*d$ is not high for isomorphism, as it does not compute an element of $[T]$.

Now suppose $\*b \ge_T \*d$ is uniformly high for descending sequences in the reticent sense.  Then $\*b$ enumerates $\omega\setminus \+O$, and since $\* d$ enumerates $\+O$, $\*b$ does too.  Thus, $\*b$ computes $\+O$.
\end{proof}

\begin{cor}
There is a degree that is uniformly high for isomorphism, but not in the reticent sense.
\end{cor}

\begin{proof}
Let $\*d$ be as in Theorem \ref{thm:separating_uniformly_high_desc_improved}, and take $\*b$ to be both PA over $\*d$ and low over $\*d$.  Then $\*b$ is uniformly high for isomorphism by Theorem~\ref{prop:pa_over_unif_desc_seq_is_unif_iso}, but not in the reticent sense, as $\*b$ is above $\*d$ but not $\+O$.
\end{proof}

\section{Randomness and genericity}\label{sec:rand_gen}

We conclude with some quick remarks on the ability of degrees that are high for isomorphism to compute generic and random sets. 

\begin{fact}
All degrees that are high for isomorphism compute a $\sigmaii$-generic since $\sigmaii$-generics form a $\sigmaii$ class, but the converse may not hold.
\end{fact}

This contrasts, sharply and unsurprisingly, with the situation for degrees that are low for isomorphism: every Cohen 2-generic is itself low for isomorphism, as is every degree computable from a Cohen 2-generic \cite{fs-lowim}.

\begin{fact}
All degrees that are high for isomorphism compute a $\piii$-random since the class of $\piii$ randoms is $\sigmaii$.
\end{fact}

Again, this contrasts with the situation for degrees that are low for isomorphism: no Martin-L\"{o}f random degree is low for isomorphism \cite{fs-lowim}.

We now use the notion of genericity to explore the relationship between degrees that are low for isomorphism and degrees that are high for isomorphism.  Recall that the triple jump of a degree high for isomorphism must compute $\+O$.  While this does not give an exact characterization of highness for isomorphism, it would be nice if it at least ruled out the degrees which are low for isomorphism.  However, this is not the case.

\begin{thm}[Kurtz~\cite{kurtz83}]
A degree $\*b$ is the $n$th jump of a weakly $(n+1)$-generic degree if and only if $\*b > \*0^{(n)}$ and $\*b$ is hyperimmune relative to $\*0^{(n)}$.
\end{thm}

It follows that the degree of $\+O$ is the $n$th jump of a weakly $n$-generic degree for every $n$ and thus the double jump of a 2-generic, and therefore in particular it is the double jump of a low-for-isomorphism degree.

\section{Summary of results and future directions}\label{sec:etc}

At the time of this writing, the following questions appear both entirely open and deeply interesting.

\begin{Q} Is there a jump complete degree which is not uniformly high for isomorphism?
\end{Q}
\begin{Q}
Is there a jump complete degree which is not uniformly high for descending sequences?
\end{Q}

\begin{Q}
Is there a jump complete degree which is not Scott complete?
\end{Q}

\begin{Q}
Is there a degree $\mathbf{d}$ which is uniformly high for isomorphism, with $\mathbf{d}'' \equiv_T \sO$?
\end{Q}

\bibliographystyle{plain}
\bibliography{recmodels,universal,random}

\end{document}